\documentclass[11pt]{article}
\pdfoutput=1

\def\pb{\pagebreak}  % \usepackage{refcheck}

\usepackage[dvipsnames]{xcolor}

\def\beq{\begin{equation} }\def\eeq{\end{equation} }\def\ep{\varepsilon}\def\1{\mathbf{1}}

\usepackage[framemethod=default]{mdframed}
\usepackage{caption}

\usepackage{indentfirst}
\usepackage{bm, mathrsfs, graphics,float,amssymb,amsmath,subeqnarray,setspace,graphicx,amsthm,epstopdf,subfigure, enumerate, color}
\usepackage[utf8]{inputenc}
\usepackage[colorlinks,
            linkcolor=red,
            anchorcolor=blue,
            citecolor=blue
            ]{hyperref}
\usepackage{natbib}

\usepackage{fullpage}
\numberwithin{equation}{section}

\newtheorem{lemma}{Lemma}
\newtheorem{theorem}{Theorem}
\newtheorem{proposition}{Proposition}
\newtheorem{definition}{Definition}

\newtheorem{remark}{Remark}

\ifx\assumption\undefined
\newtheorem{assumption}{Assumption}
\fi

\newcommand{\bT}{\top}

\newcommand{\cD}{\mathcal{D}}
\newcommand{\cO}{\mathcal{O}}
\newcommand{\cH}{\bm{\mathcal{H}}}

\newcommand{\EE}{\mathbb{E}}
\newcommand{\RR}{\mathbb{R}}

\newcommand{\xb}{\mathbf{x}}

\newcommand{\bH}{\mathbf{H}}
\newcommand{\I}{\mathbf{I}}

\newcommand{\cF}{{\bm{\mathcal{F}}}}

\newcommand{\PP}{\mathbb{P}}

\newcommand{\ud}{d}

\newcommand{\ty}{\tilde{\y}}

\newcommand{\argmin}{\mathop{\mathrm{argmin}}}

\newcommand{\minimize}{\mathop{\mathrm{minimize}}}

\newcommand{\bbx}{\bar{\x}}

\usepackage{bbm}
\newcommand{\cI}{\mathcal{I}}

\newcommand{\tC}{\tilde{C}}

\def\cK{\mathcal{K}}

\newcommand{\tO}{\tilde{\cO}}

\usepackage{multirow}
\usepackage{tablefootnote}
\usepackage{colortbl}% http://ctan.org/pkg/xcolor
\usepackage{hhline}% http://ctan.org/pkg/hhline

\usepackage{algorithm}
\usepackage{algorithmic}

\def\ep{\epsilon}

\newcommand{\vv}{\mathbf{v}}
\newcommand{\x}{\mathbf{x}}

\newcommand{\y}{\mathbf{y}}
\newcommand{\z}{\mathbf{z}}
\renewcommand{\b}{\mathbf{b}}

\newcommand{\A}{\mathbf{A}}
\newcommand{\cA}{\mathcal{A}}

\newcommand{\w}{\mathbf{w}}

\newcommand{\E}{\mathbb{E}}

\newcommand{\cP}{\bm{\mathcal{P}}}

\renewcommand{\a}{\mathbf{a}}

\newcommand{\e}{\mathbf{e}}

\newcommand{\<}{\left\langle}
\renewcommand{\>}{\right\rangle}

\usepackage{mathtools}

\def\NEON{\textsc{Neon}}
\def\SPIDER{\textsc{Spider}}
\newcommand{\bzeta}{\bm{\zeta}}

\begin{document}
\title{
Sharp Analysis for Nonconvex SGD Escaping from Saddle Points
}

\author{
Cong Fang
\thanks{
%Peking University;
email: fangcong@pku.edu.cn
}
%\thanks{This work was performed while Cong Fang was a research intern with Tencent AI Lab.}
\and
Zhouchen Lin
\thanks{
	%Tencent AI Lab;
	email:  zlin@pku.edu.cn
}
%\footnotemark[1]
\and
Tong Zhang
\thanks{
	%Tencent AI Lab;
	email: tongzhang@tongzhang-ml.org
}
%\footnotemark[2]
}
\date{\today}

\maketitle

\begin{abstract}
In this paper, we give a sharp analysis\footnote{``Sharp analysis'' does not mean that our result is the tightest. It means an improved analysis.} for  Stochastic Gradient Descent (SGD)  and  prove that SGD is able to efficiently escape from  saddle points and find an $(\ep, \cO(\ep^{0.5}))$-approximate second-order stationary point in $\tilde{\cO}(\epsilon^{-3.5})$ stochastic gradient computations for generic nonconvex optimization problems, when the objective function satisfies  gradient-Lipschitz, Hessian-Lipschitz, and dispersive noise assumptions.
This  result subverts the classical belief that  SGD requires at least $\cO(\epsilon^{-4})$ stochastic gradient computations for obtaining an $(\ep,\cO(\ep^{0.5}))$-approximate second-order stationary point.
Such SGD rate matches, up to a polylogarithmic factor of problem-dependent parameters, the rate of most accelerated  nonconvex stochastic optimization algorithms that adopt additional techniques, such as Nesterov's momentum acceleration, negative curvature search, as well as quadratic and cubic regularization tricks.
Our novel analysis gives new insights into nonconvex SGD and can be potentially generalized to a broad class of stochastic optimization algorithms.
\end{abstract}

%\tableofcontents

%\def\epH{\ep_H}
\def\epH{\delta}

\newcommand{\bchi}{\bm{\chi}}
\newcommand{\bxi}{\bm{\xi}}
\newcommand{\bgamma}{\bm{\gamma}}
\newcommand{\bpsi}{\bm{\psi}}
\newcommand{\bep}{\bm{\ep}}
\def\epH{\delta_2}
\def\radius{B}
\def\Ko{K_{o}}
\def\g{\mathbf{g}}
\newcommand{\uu}{\mathbf{u}}
\renewcommand{\a}{\mathbf{a}}
\renewcommand{\u}{\mathbf{u}}
\renewcommand{\v}{\mathbf{v}}
\newcommand{\Ib}{\mathbf{I}}
\newcommand{\Hb}{\mathbf{H}}
\def\Db{\mathbf{D}}
\newcommand{\cB}{\mathcal{B}}
\newcommand{\sK}{\mathscr{K}}
\newcommand{\cS}{\mathcal{S}}

\def\redd#1{\textcolor{orange}{#1}}

\section{Introduction}\label{sec:intro}
Nonconvex stochastic  optimization  is crucial in machine learning and have attracted tremendous attentions and unprecedented popularity.
 Lots of modern tasks that include low-rank matrix factorization/completion and principal component analysis \citep{candes2009exact,jolliffe2011principal}, dictionary learning \citep{sun2017complete}, Gaussian mixture models \citep{reynolds2000speaker}, as well as notably deep  neural networks \citep{hinton2006reducing} are formulated as nonconvex stochastic optimization problems.
In this paper, we concentrate on finding an approximate solution to the following minimization problem:
\beq\label{opt_eq}
\minimize_{\xb\in\RR^d} ~~~
f(\xb)
  \equiv
\EE_{\bzeta\sim \cD}  \left[ F(\xb; \bzeta) \right]
.
\eeq
Here, $F(\x; \bzeta)$ denotes a family of stochastic functions indexed by some random variable $\bzeta$ that obeys some prescribed distribution $\cD$, and we consider the  general case where $f(\x)$ and $F(\x;\bzeta)$ have Lipschitz-continuous gradients and Hessians and might be \textit{nonconvex}. 
In empirical risk minimization tasks, $\bzeta$ is an  uniformly discrete distribution over the set of training sample indices, and the stochastic function $F(\xb; \bzeta)$ corresponds to the nonconvex loss associated with such a sample.

One of the classical algorithms for optimizing \eqref{opt_eq} is the \textbf{Stochastic Gradient Descent} (SGD) method, which performs descent updates iteratively via the inexpensive stochastic gradient $\nabla F(\x; \bzeta)$ that serves as an unbiased estimator of (the inaccessible) gradient $\nabla F(\x)$ \citep{robbins1951stochastic,bottou2008tradeoffs}, i.e.~$\EE_{\bzeta\sim \cD} \left[ \nabla f(\xb;\bzeta) \right]  =  \nabla f(\xb)$.
Let $\eta$ denote the positive  stepsize, then at steps $t=1,2,\ldots$, the iteration performs the following update:
\beq\label{SGD}
\xb^t 
= 
\xb^{t-1} - \eta \nabla F(\xb^{t-1}; \bzeta^{t})
,
\eeq
where $\bzeta^{t}$ is randomly sampled at iteration $t$.
SGD admits perhaps the simplest update rule among stochastic first-order methods.
See Algorithm \ref{algo:SGDmeta} for a formal illustration of the meta algorithm.
It has gained tremendous popularity due to its exceptional practical performance.
Taking the example of training deep neural networks, the dominating algorithm at present time is SGD \citep{abadi2016tensorflow}, where the stochastic gradient is computed via one backpropagation step.
Superior characteristics of SGD have been observed in many empirical studies, including but \textit{not} limited to fast convergence, desirable solutions of low training loss, as well as its generalization ability.

Turning to the theoretical side, relatively mature and concrete analysis in existing literatures \citet{rakhlin2012making, agarwal2009information} show that SGD achieves an \textit{optimal} rate of convergence for convex objective function under some standard regime.
Specifically, the convergence rate   of $\cO(1/T)$ in term of the function optimality gap match the algorithmic lower bound for an appropriate class of strongly convex functions \citep{agarwal2009information}.

Despite  the optimal convex optimization rates that SGD achieves, the provable \textit{nonconvex} SGD convergence rate result has long stayed upon on finding an $\ep$-approximate first-order stationary point $\xb$: with high probability SGD finds an $\x$ such that $\|\nabla f(\x)\| \le \ep$ in $\cO(\ep^{-4})$ stochastic gradient computational cost under the gradient Lipschitz condition of $f(\x)$ \citep{NESTEROV}.
In contrast, our goal in this paper is to find an \textit{$(\ep,\sqrt{\rho\ep})$-approximate second-order stationary point} $\x$ such that $\|\nabla f(\x)\| \le \ep$ and the least eigenvalue of the Hessian matrix $\nabla^2 f(\x)$ is $\ge -\sqrt{\rho \ep}$, where $\rho>0$ denotes the so-called Hessian-Lipschitz parameter to be specified later \citep{nesterov2006cubic,tripuraneni2018stochastic,carmon2018accelerated,agarwal2017finding}.
Putting it differently, we need to escape from all first-order stationary points that admit a strong negative Hessian eigenvalue (a.k.a.~saddle points) \citep{dauphin2014identifying} and lands at a point that quantitatively resembles a local minimizer in terms of the  gradient norm and least Hessian eigenvalue.

\begin{algorithm}[tb]
	\caption{SGD (Meta version)}
	\begin{algorithmic}[1]
	\STATE {\bf for} $t =1,2,\dots$ {\bf do}
	\STATE \quad
	Draw an independent $\bzeta^{t}\sim \cD$ and set 
	$\x^t \leftarrow \x^{t-1} - \eta \nabla F(\x^{t-1}; \bzeta^{t})$
				{\hfill $\diamond$ SGD step} 
	\STATE \quad {\bf if} {Stopping criteria is satisfied} {\bf then}
			\STATE \quad\quad \textbf{break}
	\end{algorithmic}
\label{algo:SGDmeta}
\end{algorithm}

Results on the convergence rate of SGD for finding an $(\ep,\sqrt{\rho\ep})$-approximate second-order stationary point have been scarce until very recently.\footnote{Some authors work with $(\ep,\delta)$-stationary point and we ignore such expression due to the natural choice $\delta = \sqrt{\rho\ep}$ in optimization literature \citep{nesterov2006cubic,jin2017escape}.}
To the best of our knowledge, \citet{ge2015escaping} provided the first theoretical result that SGD with \textit{artificially injected spherical noise} can escape from all saddle points in  polynomial time.
Moreover, \citet{ge2015escaping} showed that SGD finds an $(\ep,\sqrt{\rho\ep})$-approximate second-order stationary point at a stochastic gradient computational cost of $\tilde{\cO}(\text{poly}(d)\epsilon^{-8})$.
A recent follow-up work by \cite{daneshmand2018escaping} derived a convergence rate of $\tO(d\epsilon^{-10})$ stochastic gradient computations.
These milestone works \cite{ge2015escaping, daneshmand2018escaping} showed that SGD can always escape from saddle points and can find an approximate local solution of \eqref{opt_eq} with a stochastic gradient computational cost that is polynomially dependent on problem-specific parameters.
Motivated by these recent works, the current paper tries to answer the following questions:
\textit{
\begin{enumerate}[(i)]
\item
Is it possible to sharpen the analysis of SGD algorithm and obtain a reduced stochastic gradient computational cost for finding an $(\ep,\cO(\ep^{0.5}))$-approximate second-order stationary point?
\item
Is artificial noise injection absolutely necessary for SGD to find an approximate second-order stationary point with an \textit{almost} dimension-free stochastic gradient computational cost? 
\end{enumerate}
}

To answer aforementioned question (i), we provide a \textit{sharp} analysis and prove that SGD with variants only on stopping criteria finds an $(\ep,\sqrt{\rho\ep})$-approximate stationary point at a remarkable $\tilde{\cO}(\ep^{-3.5})$ stochastic gradient computational cost for solving \eqref{opt_eq}. This is a  unexpected result because
it has been conjectured by many \citep{xu2018first,allen2018neon2,tripuraneni2018stochastic} that an $\tO(\ep^{-4})$ cost is required to find an $(\ep,\sqrt{\rho\ep})$-approximate second-order stationary point. Our result on SGD \textit{negates} this conjecture and serves as the sharpest stochastic gradient computational cost for SGD prior to this work.
To answer question (ii) above, we propose a novel \textit{dispersive noise assumption} and prove that under such an assumption, SGD requires \textit{no} artificial noise injection in order to achieve the aforementioned sharp stochastic gradient computational cost.
Such noise assumption is satisfied in the case of infinite online samples and Gaussian sampling zeroth-order optimization, and can be satisfied automatically by injecting artificial ball-shaped, spherical uniform, or Gaussian noises.

We emphasize that the $\tilde{\cO}(\ep^{-3.5})$ stochastic gradient computational cost is, however, \textit{not} the lower bound complexity for finding an $(\ep,\sqrt{\rho\ep})$-approximate second-order stationary point for  problem \eqref{opt_eq}.
Recently, \cite{fang2018spider} applied a novel variance reduction technique named \SPIDER\ \textit{tracking} and proposed the \SPIDER-SFO\textsuperscript{+} algorithm which  achieves a stochastic gradient computational cost of $\tilde{\cO}(\ep^{-3})$ for finding an $(\ep,\sqrt{\rho\ep})$-approximate second-order stationary point.
It is our belief that variance reduction techniques are necessary to achieve a stochastic gradient computational cost that is strictly sharper than $\tilde\cO(\ep^{-3.5})$.  We also note that the promising $\tilde{\cO}(\ep^{-3.5})$ complexity relies on the Hessian-smooth assumption,  whereas the standard $\cO(\ep^{-4})$ complexity for searching an approximate first-order stationary point does not need this assumption.

\subsection{Our Contributions}
We study theoretically in this work the SGD algorithm for minimizing nonconvex function $\EE[F(\x;\bzeta)]$.
Specially, this work contributes the following:

\begin{enumerate}[(i)]
\item 
We propose a sharp convergence analysis for the classical and simple SGD and prove that the total stochastic gradient computational cost to find a second-order stationary point is at most $\tO(\ep^{-3.5})$ under both Lipschitz-continuous gradient and Hessian assumptions of objective function.
Such convergence rate matches the most accelerated nonconvex stochastic optimization results that such as Nesterov's momentum acceleration, negative curvature search, and quadratic and cubic regularization tricks.

\item
We propose the \textit{dispersive noise assumption} and prove that under such an assumption, SGD ensures to escape all saddles that has a strongly negative Hessian eigenvalue.  Such type of noise generalizes the existing artificial ball-shaped noise and is widely applicable to many tasks.

\item
Our novel analytic tools for proving saddle escaping and fast convergence of SGD is of independent interests, and they shed lights on developing and analyzing  new  stochastic optimization algorithms.
\end{enumerate}

\paragraph{Organization}
The rest of the paper is organized as follows.
\S\ref{sec:main} provides the SGD algorithm and the main convergence rate theorem for finding an $(\ep,\sqrt{\rho\ep})$-approximate second-order stationary point. Related Works are discussed in \S\ref{sec:compar}. We conclude our paper in \S\ref{sec:conclusion} with proposed future directions. In Appendix \ref{sec:proof_tech}, we sketch the proof of our convergence rate theorem by providing and discussing three core propositions. And all the missing proofs are detailed  in the   Appendix rest sections.

\paragraph{Notation}
Let $\|\cdot\|$ denote the Euclidean norm of a vector or spectral norm of a square matrix.
Denote $p_n = \cO(q_n)$ for a sequence of vectors $p_n$ and positive scalars $q_n$ if there is a global constant $C$ such that $|p_n| \le Cq_n$, and $p_n = \tilde{\cO}(q_n)$ such $\tC$ hides a poly-logarithmic factor of $d$ and $\ep$.
Denote $p_n = \tilde{\Omega} (q_n)$ if there is $\tC$ which hides a poly-logarithmic factor such that $| p_n|\geq \tC q_n$. We denote $p_n \asymp q_n$ if there is  $\tC$ which hides a poly-logarithmic factor of $d$ and $\ep$ such that $p_n = \tC q_n$.  Further, we denote linear transformation of set $\cA\subseteq \RR^d$ as $c_1 + c_2\cA := \{c_1 + c_2 a: a\in\cA\}$.
Let $\lambda_{\min}(\A)$ denote the least eigenvalue of a real symmetric matrix $\A$.  We denote $\cB(\x, R)$ as the $R$-neighborhood of $\x^0$, i.e. the set  $\{\y\in \RR^d  : \| \y - \x^0\|\leq R  \}$.

%For fixed $K \ge k \ge 0$, let $\x_{k:K}$ denote the sequence $\{\x^k ,\dots , \x^K \}$.
%Other notations are explained at their first appearance.

%%% Local Variables:
%%% mode: latex
%%% TeX-master: t
%%% End:

%  LocalWords:  nonconvex candes jolliffe reynolds hinton Lipschitz
%  LocalWords:  Hessians indices eq SGD robbins bottou tradeoffs jin
%  LocalWords:  stepsize abadi tensorflow backpropagation literatures
%  LocalWords:  rakhlin agarwal NESTEROV nesterov tripuraneni carmon
%  LocalWords:  ge xu allen online SFO Nesterov's

\section{Algorithm and Main Result}\label{sec:main}
In this section, we formally state  SGD  and the corresponding convergence rate theorem.
In \S\ref{sec:assu}, we propose the key assumptions for the objective functions and noise distributions.
In \S\ref{sec:SGDmain}, we detail  SGD in Algorithm \ref{algo:SGD} and present the main convergence rate theorem.

\subsection{Assumptions and Definitions}\label{sec:assu}

\begin{assumption}[Smoothness]\label{assu:smooth}
We assume that the objective function satisfies some smoothness\footnote{The smoothness gradient condition for $F(\xb;\bzeta)$ is only needed for searching an approximate second-order stationary point. To find a first stationary point, we only can replace \eqref{GradSmooth} with a relaxed one: $\|\nabla f(\xb) - \nabla f(\xb') \|
 \le 
L \|\xb-\xb'\|$.
} conditions: for all $\xb,\xb' \in \RR^d$, we have
\beq\label{GradSmooth}
\|\nabla F(\xb;\bzeta) - \nabla F(\xb';\bzeta) \|
 \le 
L \|\xb-\xb'\|
 ,
\eeq
and
\beq\label{HessSmooth}
\|\nabla^2 f(\xb) - \nabla^2 f(\xb') \|
 \le 
\rho \|\xb-\xb'\|
 .
\eeq
\end{assumption}

With Hessian-Lipschitz parameter $\rho$ prescribed in \eqref{HessSmooth}, we formally define the $(\ep,\sqrt{\rho\ep})$-approximate second-order stationary point.
To best of our knowledge, such concept firstly appeared in \cite{nesterov2006cubic}:

\begin{definition}[Second-order Stationary Point]\label{defi:SSP}
Call $\xb\in \RR^d$ an \textit{$(\ep,\sqrt{\rho\ep})$-approximate second-order stationary point} if
\[
\|\nabla f(\xb)\| \le \epsilon
 ,\qquad  
\lambda_{\min} (\nabla^2 f(\xb))  \ge -\sqrt{\rho \ep}
.
\]
\end{definition}

Let the starting point of our SGD algorithm be $\tilde{\x}\in\RR^d$.
We assume the following boundedness assumption:

\begin{assumption}[Boundedness]\label{assu:main}
The $\Delta := f(\tilde{\x})- f^* < \infty$ where $f^* = \inf_{\x\in\RR^d} f(\x)$ is the global infimum value of $f(\x)$.
\end{assumption}

Turning to the assumptions on noise, we first assume the following:
\begin{assumption}[Bounded Noise]\label{assu:noise}
For any $\x\in\RR^d$, the stochastic gradient $\nabla F(\xb; \bzeta)$ satisfies:
\begin{eqnarray}\label{bounded}
\left\| \nabla F(\x, \bzeta) -\nabla f(\x) \right\|^2\leq \sigma^2
,\qquad a.s.
\end{eqnarray}

An alternative (slighter weaker) assumption that also works is to assume  that the norm  of noise satisfies subgaussian distribution, i.e. for any $\x\in\RR^d$,

\begin{eqnarray}\label{bound2}
\EE_{\bzeta} \left[\exp(\left\| \nabla F(\x; \bzeta) -\nabla f(\x) \right\|^2/\sigma^2)\right]\leq 1
.
\end{eqnarray}

\end{assumption}

Assumptions \ref{assu:smooth}, \ref{assu:main} and \ref{assu:noise} are standard in nonconvex optimization literatures \citep{ge2015escaping,xu2018first,allen2018neon2,fang2018spider}. We treat the parameters $L$, $\rho$, $\Delta$, and $\sigma$ as global constants, and focus on the dependency for stochastic gradient complexity on $\ep$ and $d$.

For the purpose of fast saddles escaping, we need an extra noise shape assumption.
Let $q^*$ be a positive real, and let $\v$ be a unit vector.
We define a set property as follows:

\begin{definition}[$(q^*,\v)$-narrow property]\label{defi:qnarrow}
We say that a Borel set $\cA\subseteq \RR^d$ satisfies the $(q^*,\v)$-\textbf{narrow property}, if for any $\u \in \cA$ and $q\ge q^*$, $\u+q\v \in \cA^c$ holds, where $\cA^c$ denotes the complement set of $\cA$.
\end{definition}

It is easy to verify that the first parameter in the narrow property is linearly scalable and translation invariant with sets, i.e.~if $\cA$ satisfies $(q^*,\v)$-narrow property, then for any $c_1\in\RR^d$ and $c_2\in\RR$, $c_1 + c_2\cA$ satisfies the $(|c_2|q^*,\v)$-narrow property. Next, we introduce the $\v$-dispersive property as follows:

\begin{definition}[$\v$-dispersive property]\label{defi:vdispersive}
Let $\tilde{\bxi}$ be a  random vector satisfying Assumption \ref{assu:noise}.
We say that $\tilde{\bxi}$ has the $\v$-\textbf{dispersive property}, if for an arbitrary set $\cA$ that satisfies the $(\sigma / (4\sqrt{d}), \v )$-narrow property (as in Definition \ref{defi:qnarrow}) the following holds:
\beq\label{smallprob1}
\PP\left( \tilde{\bxi} \in \cA \right)
 \le \frac14
.
\eeq
\end{definition}
Obviously, if $\tilde{\bxi}$ satisfies  $\v$-dispersive property, for any fixed vector $\a$,  then $\tilde{\bxi}+\a$ also  satisfies  $\v$-dispersive property. We then present the dispersive noise assumption as follows:

\begin{assumption}[Dispersive Noise]\label{assu:noise2}
For an arbitrary point $\xb\in \RR^d$, $\nabla f(\xb; \bzeta)$ admits the $\v$-dispersive property (as in Definition \ref{defi:vdispersive}) \textbf{for any unit vector $\v$}.
\end{assumption}

Assumption \ref{assu:noise2} is motivated from the key lemma for escaping from saddle points in \citet{jin2017escape}, which obtains a sharp rate for gradient descent escaping from saddle points.
Such an assumption  enables SGD to move out of a \textit{stuck region} with probability $\ge 3/4$ in its first step and enables escaping from saddle points (by repeating  logarithmic rounds).
We would like to emphasize that the $\v$-dispersive noises contain many canonical examples; see the following

\paragraph{Examples of Dispersive Noises}
Here we exemplify a few noise distributions that satisfy the $\v$-dispersive property, that is, for an arbitrary set $\cA$ with $(q^*,\v)$-narrow property, where $q^* = \sigma / (4\sqrt{d})$.
We have the following proposition:

\begin{proposition}\label{prop:noise}
For the following noise distributions, \eqref{smallprob1} in Definition \ref{defi:vdispersive} is satisfied:
\begin{enumerate}[(i)]
\item
\textit{Gaussian noise}:
$\tilde\bxi = \sigma/\sqrt{d} * \bchi$ where $\bchi$ is the standard Gaussian noise with covariance matrix  $\Ib_d$;
\item
\textit{Uniform ball-shaped noise or spherical noise}:
$\tilde\bxi = \sigma * \bxi_b$, where $\bxi_b$ is uniformly sampled from the unit ball centered at ${\bf 0}$;
\item
\textit{Artificial noise injection}:
$\tilde\bxi = \nabla f(\xb;\bzeta) + \tilde\bgamma$, where $\tilde{\bgamma}$ is some independent artificial noise that is $\v$-dispersive for any $\v$.
\end{enumerate}
\end{proposition}
The proof of Proposition \ref{prop:noise} is shown in Appendix \ref{sec:proof,prop:noise}.

\subsection{SGD  and Main Theorem}\label{sec:SGDmain}
Our SGD algorithm for analysis purposes is detailed in Algorithm \ref{algo:SGD}. Our SGD algorithm only differs from classical SGD algorithms  on  stopping criteria.
Distinct from the classical ones that simply terminate in a certain number of steps and output the final iterate or a randomly drawn iterate, the SGD  we consider here introduces a \textit{ball-controlled mechanism} as the stopping criteria: if $\x^k$ exits a small neighborhood in $K_0$ iterations (Line \ref{line:stop_criteria} to  \ref{line:break}), one starts over and do the next round of SGD;
if exiting does \textit{not} occur in $K_0$ iterations, then the algorithm simply outputs an arithmetic average of $\x^k$ of the last $K_0$ iterates within the neighborhood, which in turns is an $(\ep,\sqrt{\rho\ep})$-approximate second-order stationary point with high probability.
In contrast with the stopping criteria in the deterministic setting that checks the descent in function values \citep{jin2017escape}, the function value in stochastic setting is reasonably costly to approximate (costs  $O(\ep^{-2})$ stochastic gradient computations), and the error plateaus might be hard to observe theoretically.

\begin{algorithm}[tb]
	\caption{SGD (For finding an $(\ep,\sqrt{\rho\ep})$-approximate second-order stationary point): Input $\tilde{\x}$, $K_0\asymp \ep^{-2}$, $\eta \asymp \ep^{1.5}$, and $\radius \asymp \ep^{0.5}. $}
	\begin{algorithmic}[1]
	\STATE Set $t = 0$, $k=0$, $\x^0 = \tilde{\x}$
    \WHILE{$k < K_0$} \label{line:stop_criteria}
		\label{line3}
			\STATE
				Draw an independent $\bzeta^{k+1}\sim \cD$ and set 
				$\x^{k+1} \leftarrow \x^{k} - \eta \nabla F(\x^{k}; \bzeta^{k+1})$
				{\hfill $\diamond$ SGD step} 
			\STATE
				$t \leftarrow t+1, ~ k \leftarrow k+1$
				{\hfill $\diamond$ Counter of SGD steps}
		 \label{line:SGDupdate}
		\IF{$\| \x^{k} - \x^0 \|> \radius$ }
			\STATE $\x^0 \leftarrow \x^k$, $k\leftarrow 0$
		\label{line:break}
		\ENDIF

	\ENDWHILE 	\label{line10}
	\STATE
		$\bar{\x}_{output} \leftarrow (1 / K_0) \sum_{k=0}^{K_0-1} \x^k$
							{\hfill $\diamond$ Reach this line in $t \le T_0 = \tO(\ep^{-3.5})$ SGD steps w.h.p.}
	\RETURN $\bar{\x}_{output}$
							{\hfill $\diamond$ Return an $(\ep,\sqrt{\rho\ep})$-approximate second-order stationary point}
	\end{algorithmic}
	
	\label{algo:SGD}
\end{algorithm}

\paragraph{Parameter Setting} We set the hyper-parameters\footnote{  Set  $\tilde{\eta} = \frac{B^2 \delta}{512 \max(\sigma^2, 1) \log(48p^{-1}) \left\lfloor\frac{\log(3\cdot p^{-1})}{\log(0.7)}+1 \right\rfloor \log(d)}\asymp \ep^{1.5}$. Because $\eta$ in \eqref{parametersetting}  involves logarithmic factors on $K_0$ and $\tC_1$,  a simple choice to set   the step size is  as $\eta =\tilde{\eta} \log^{-3}(\tilde{\eta}^{-1})\asymp \ep^{1.5}$.  } for Algorithm \ref{algo:SGD} as follows:
\begin{eqnarray}\label{parametersetting}
&&\tC_1 =  2 \left\lfloor\frac{\log(3 \cdot p^{-1})}{\log(0.7)}+1 \right\rfloor\log\left(\frac{24\sqrt{d}}{\eta}\right)\asymp 1, \quad \delta = \sqrt{\rho \ep} \asymp \ep^{0.5},\notag\\
&&\delta_2 = 16 \delta\asymp\ep^{0.5},  \quad B  = \frac{ \delta}{\rho\tC_1}\asymp \ep^{0.5},  \quad K_0 = \tC_1 \eta^{-1}\delta_2^{-1}\asymp\ep^{-2},\notag\\
 &&  \eta \leq \frac{B^2\delta}{64\max(\sigma^2,1) \tC_1 \log(48K_0/p)  } \cdot \frac{1}{ 3+ \log(K_0) }\asymp\ep^{1.5}.
\end{eqnarray}
 For brevity of analysis, we assume $B\leq \min(1, \frac{\sigma}{L}, \frac{1}{L})\asymp \cO(1)$, and $\delta\leq 1$.  In other words, we assume the accuracy $\ep \leq \cO(1)$.

Now we are ready to present our main result of SGD theorem.

\begin{theorem}[SGD  Rate]\label{theo:main}
Let Assumptions \ref{assu:smooth}, \ref{assu:main}, \ref{assu:noise}, and \ref{assu:noise2} hold.
Let the parameters $K_0$, $\eta$ and $B$ be set in \eqref{parametersetting} with   $p \in (0,1)$ being the  error probability, and set $T_1 = \left\lceil \frac{7\Delta \eta K_0}{B^2}\right\rceil+1 \asymp\ep^{-1.5}$, then  running Algorithm \ref{algo:SGD}  in $T_0 = T_1\cdot K_0\asymp\frac{\Delta\rho^{1/2}}{\max(\sigma^2,1)\ep^{3.5}}\asymp\ep^{-3.5}$, with probability at least $1-(T_1+1)  \cdot p$,  SGD  outputs an $\bar{\x}_{output}$ satisfying 
\beq\label{ESSP}
\|\nabla f(\bar{\x}_{output})\| \le  18 \rho B^2\asymp \ep
 ,\qquad
\lambda_{\min}(\nabla^2 f(\bar{\x}_{output})) \ge - 17\delta \asymp  -\sqrt{\rho\ep}
.
\eeq
Treating $\sigma$, $L$, and $\rho$ as global constants, the stochastic gradient computational cost is $\tO(\ep^{-3.5})$.	
\end{theorem}

 Strikingly, Theorem \ref{theo:main} indicates that SGD in Algorithm \ref{algo:SGD} achieves a stochastic gradient computation cost of $\tO(\ep^{-3.5})$ to find an $(\ep, \sqrt{\rho\ep})$-approximate second-order stationary point\footnote{For searching a more general  $(\ep_g,\ep_H)$-approximate second-order stationary point, one can obtain an complexity of $\tO(\ep_g^{-3.5}+\ep_H^{-7})$ using the same technique. }.
Compared with existing algorithms that achieves an $\tO(\ep^{-3.5})$ convergence rate, SGD is comparatively simpler to implement and does \textit{not} invoke any additional techniques or iterations such as momentum acceleration \citep{jin2018accelerated}, cubic regularization \citep{tripuraneni2018stochastic}, regularization \citep{allen2018make}, or \NEON-type negative curvature search \citep{xu2018first,allen2018neon2}.  

\begin{algorithm}[!tb]
	\caption{Noise-Scheduled SGD (For finding an $(\ep,\sqrt{\rho\ep})$-approximate second-order stationary point): Input $\tilde{\x}$, $K_o= 2\log\left( \frac{24\sqrt{d}}{\eta}\right)\eta^{-1}\delta_2^{-1}\asymp \ep^{-2}$, $K_0\asymp \ep^{-2}$, $\eta \asymp \ep^{1.5}$, and $\radius \asymp \ep^{0.5}. $}
	\begin{algorithmic}[1]
		\STATE Set $t = 0$, $k=0$, $\x^0 = \tilde{\x}$
		\WHILE{$k < K_0$}
		\label{line:stop_criteria2}
		%		\label{line3}
		\IF{$\text{mod}(k,K_o)=0$}
		\STATE Draw an independent $\bzeta^{k+1}\sim \cD$ and Gaussian noise $\bxi_g \sim N(0,(\sigma^2/d) \Ib_d)$ 
		\\ \hspace{.2in}
		$\x^{k+1} \leftarrow \x^{k+1} - \eta \left( \nabla F(\x^k; \bzeta^{k+1}) + \bxi_g \right)$
		{\hfill $\diamond$ SGD step (with noise injection)} 
		\ELSE
		\STATE
		Draw an independent $\bzeta^{k+1}\sim \cD$ and set 
		$\x^{k+1} \leftarrow \x^{k} - \eta \nabla F(\x^{k}; \bzeta^{k+1})$
		{\hfill $\diamond$ SGD step} 
		\ENDIF
		\STATE
		$t \leftarrow t+1$, $k \leftarrow k+1$
		{\hfill $\diamond$ Counter of SGD steps}
		\label{line:SGDupdate2}
		\IF{$\| \x^k - \x^0 \|> \radius$ }
		\STATE $\x^0 \leftarrow \x^k$, $k \leftarrow 0$
		%		\label{denote ck}
		%		\label{line:break}
		\ENDIF
		%		\label{line10}
		\ENDWHILE
		\STATE
		$\bar{\x}_{output} \leftarrow (1 / K_0) \sum_{k=0}^{K_0-1} \x^k$
		{\hfill $\diamond$ Reach this line in $t \le T_0 = \tO(\ep^{-3.5})$ SGD steps, w.h.p.}
		\RETURN $\bar{\x}_{output}$
		{\hfill $\diamond$ Return an $(\ep,\sqrt{\rho\ep})$-approximate second-order stationary point}
	\end{algorithmic}
	\label{algo:SGD2}
\end{algorithm}

Admittedly, the best-known SGD theoretical guarantee in Theorem \ref{theo:main} relies on a dispersive noise assumption.
To remove such an assumption,  we argue that only $\tO(1)$ steps of each round does one need to run an SGD step of dispersive noise to enable efficient escaping.
We propose a variant of SGD called \textit{Noise-Scheduled SGD} which requires artificial noise injection but does \textit{not}  rely on a dispersive noise assumption. The algorithm  is shown  in Algorithm \ref{algo:SGD2}.  One can obtain the  convergence property straightforwardly.

\begin{remark}
For the function class that admits the \textit{strict-saddle} property \citep{carmon2018accelerated,ge2015escaping,jin2017escape}, an approximate second-order stationary point is guaranteed to be an approximate local minimizer.
For example for optimizing a $\sigma_*$-strict-saddle function, one can first find an $(\ep_*,\sqrt{\rho\ep_*})$-approximate second-order stationary point with $\ep_* \le \sigma_*^2/(2\rho)$ which is guaranteed to be an approximate local minimizer due to the strict-saddle property.
Our SGD convergence rate $\cO(\ep_*^{-3.5}) = \cO(\sigma_*^{-7})$ is independent of the target accuracy $\ep$, and one can run a standard convex optimization theory to obtain an $\cO(1/t)$ convergence rate in terms of the optimality gap.
Limited by space we omit the details.
\end{remark}

%  LocalWords:  SGD Lipschitz Hessians HessSmooth infimum subgaussian
%  LocalWords:  nonconvex literatures ge xu allen Borel jin smallprob
%  LocalWords:  online parametersetting carmon

\section{Discussions on Related Works}\label{sec:compar}
Due to the recent heat of deep learning, many researchers have studied  the nonconvex SGD method from various perspectives in the machine learning community.
We compare our results with concurrent theoretical works on  nonconvex SGD in the following discussions.
For clarity, we also compare the convergence rates of some works most related to ours in Table~\ref{tab:comparison}.

\begin{table}[!tb]
	\centering
	\begin{tabular}{|l|ll|c|}
		\hline
		&Algorithm			&						&	SG Comp.~Cost
		\\ \hline
		\multirow{6}{*}{\begin{tabular}{@{}l@{}}
				SGD Variants
		\end{tabular}}
		&\NEON+SGD			&\citep{xu2018first}			&	\multirow{2}{*}{\begin{tabular}{@{}l@{}}
				$\ep^{-4}$
		\end{tabular}}
		\\ 
		&\NEON2+SGD		&\citep{allen2018neon2}		&
		\\ \cline{2-4} 
		&Stochastic Cubic		&\citep{tripuraneni2018stochastic}
		&	\multirow{2}{*}{\begin{tabular}{@{}l@{}}
				$\ep^{-3.5}$
		\end{tabular}}
		\\
		&RSGD5				&\citep{allen2018make}		&
		\\ \cline{2-4} 
		&Natasha2$^\Delta$		&\citep{allen2018natasha2}	&	\multirow{2}{*}{\begin{tabular}{@{}l@{}}
				$\ep^{-3.5}$
		\end{tabular}}
		\\ 
		&\NEON2+SNVRG$^\Theta$
		&\citep{zhou2018finding}		&
		\\ 
		&\SPIDER
		& \citep{fang2018spider}
		&\cellcolor{orange!25}	
				 $\ep^{-3}$

	%	\\ 
%		&\multirow{2}{*}{\begin{tabular}{@{}l@{}}
%				\SPIDER+Cubic
%		\end{tabular}}		
%		&\citep{zhou2019stochastic}$^\Lambda$&	\cellcolor{orange!25}$\ep^{-3}$
%		\\
%		&&\citep{shen2019stochastic}$^\Lambda$&	\cellcolor{orange!25}$\ep^{-3}$
		\\
		
		\hline
		\multirow{3}{*}{\begin{tabular}{@{}l@{}}
				Original SGD
		\end{tabular}}
		&\multirow{4}{*}{
			\begin{tabular}{@{}l@{}}
				SGD
		\end{tabular}}
		&\citep{ge2015escaping}		&	$poly(d) \ep^{-8}$
		\\ 
		&					&\citep{daneshmand2018escaping}
		&	$d^4\ep^{-5}$
		\\ 
	&	&\citep{jin2019stochastic}
		&	$\ep^{-4}$
		\\ 
		&					&(this work)				&	\cellcolor{yellow!25} $\ep^{-3.5}$
		\\ \hline
	\end{tabular}
	\caption{Comparable results on the stochastic gradient computational cost for nonconvex optimization algorithms in finding an $(\epsilon, \sqrt{\rho\epsilon})$-approximate second-order stationary point for problem \eqref{opt_eq} under standard assumptions.
		Note that each stochastic gradient computational cost may hide a poly-logarithmic factors of $d$, $n$, $\ep$.
		\\
		{\footnotesize Orange-boxed: \SPIDER\ reported in orange-boxed is the only existing variant stochastic algorithm that achieves provable faster rate by order than simple SGD.}
		\\
		{\footnotesize $^\Delta$: \citet{allen2018natasha2} also obtains a stochastic gradient computational cost of $\tO(\ep^{-3.25})$ for finding a relaxed $(\epsilon, \cO(\epsilon^{0.25}))$-approximate second-order stationary point.}
		\\
		{\footnotesize $^\Theta$: With additional third-order smoothness assumptions, SNVRG \citep{zhou2018finding} achieves complexity of $\tO(\ep^{-3})$.}\\
	}
	\label{tab:comparison}
	
\end{table}

\begin{enumerate}[(i)]
\item \textbf{Pioneer SGD:}
The first work on SGD escaping from saddle points \cite{ge2015escaping} obtain a stochastic gradient computational cost of $\tO(\text{poly(d)}\ep^{-8})$.%
\footnote{%
The analysis in \citep{ge2015escaping} indicates a $poly(d)$ factor of $\cO(d^8)$ at least.
}
Later, \citet{jin2017escape,jin2018accelerated} noise-perturbed GD and AGD and achieve sharp gradient computational costs, which suggests the possibility of sharper SGD rate for escaping saddles.
Our analysis in this work is partially motivated by \citet{jin2017escape} for escaping from saddle points, but generalizes the noise condition and needs \textit{no} deliberate noise injections which is \textit{not} the original GD/SGD algorithm in a strict sense.

\item \textbf{Concurrent SGD:}
A recent result by \citet{daneshmand2018escaping} obtains a stochastic computation cost of $\tO(\tau^{-2}\ep^{-10})$ to find an $(\ep,\sqrt{\rho\ep})$-approximate second-order stationary point.
The highlight of their work is that they need \textit{no} injection of artificial noises.
Nevertheless in their work, the Correlated Negative Curvature parameter $\tau^{-2}$ \textit{cannot} be treated as an $\cO(1)$-constant.
Taking the case of injected spherical noise or Gaussian noise, it can be at most linearly dependent on $d$ [Assumption \ref{assu:noise2}], so the result is  \textit{not} (almost) dimension-free, and  worst-case convergence rate shall be interpreted  as $\tO(d^4\ep^{-5})$.  Concurrently with our work, \citet{jin2019stochastic} also  extend the technique in \cite{jin2017escape} to work on SGD and prove that SGD with the injected noise can find an approximate second-order stationary point with stochastic computation cost of $\tO(\ep^{-4})$. Besides, they  further study the case when the individual function $F(\x;\bzeta)$ does not satisfy gradient-smooth  condition and obtain a complexity of   $\tO(d\ep^{-4})$.

\item\textbf{NC search + SGD:}
The \NEON+SGD \citep{xu2018first,allen2018neon2} methods achieve a dimension-free convergence rate of $\tO(\ep^{-4})$ for the general problem of form \eqref{opt_eq} to reach an $(\ep, \sqrt{\rho\ep})$-approximate second-order stationary point.
Prior to this, classical nonconvex GD/SGD only achieves such a rate for finding an $\ep$-approximate first-order stationary point \citep{NESTEROV}, which, with the help of \NEON\ method, successfully escapes from saddles via a Negative Curvature (NC) search iteration \citep{xu2018first,allen2018neon2}.

\item \textbf{Regularization + SGD:}
Very recently, \citet{allen2018make} takes a quadratic regularization approach and equips it with a negative-curvature search iteration \NEON2 \citep{allen2018neon2}, which successfully improves the rate to $\tO(\ep^{-3.5})$.
In comparison, our method achieves essentially the same rate without using  regularization methods.
\citet{tripuraneni2018stochastic} proposed a stochastic variant of cubic regularization method \citep{nesterov2006cubic,agarwal2017finding} and achieves the same $\tilde\cO(\ep^{-3.5})$ convergence rate, which is the first achieving such rate without invoking variance reduced gradient techniques.%
\footnote{%
Note in the convergence rate here, we also includes the number of stochastic Hessian-vector product evaluations, each of which takes about the same magnitude of time as per stochastic gradient evaluation.}

\item \textbf{NC search + VR:}
\citet{allen2018natasha2} converted a NC search method to the online stochastic setting \citep{carmon2018accelerated} and achieved a convergence rate of $\tO(\ep^{-3.5})$ for finding an $(\ep,\sqrt{\rho\ep})$-approximate second-order stationary point.  For finding a relaxed $(\epsilon, \cO(\epsilon^{0.25}))$-approximate second-order stationary point, \citet{allen2018natasha2}  obtains a lower stochastic gradient computational cost of $\tO(\ep^{-3.25})$.
With a recently proposed \textit{optimal} variance reduced gradient techniques applied, \SPIDER\ achieves the state-of-the-art $\tilde{\cO} (\ep^{-3})$ stochastic gradient computational cost \citep{fang2018spider}.%
\footnote{The independent work \citet{zhou2018finding} achieves a similar convergence rate for finding an $\ep$-approximate second-order stationary point by imposing a third-order smoothness conditions on the objective.
} Very recently, \citet{zhou2019stochastic} and \citet{shen2019stochastic} have independently designed  powerful cubic algorithms using~\SPIDER~techinique and also obtained a complexity of $\tO(\ep^{-3})$.

\end{enumerate}

\subsection{More Related Works}

\paragraph{VR  Methods}
In the recent two years, sharper convergence rates for nonconvex stochastic optimization can be achieved using \textit{variance reduced gradient techniques} \citep{schmidt2017minimizing, johnson2013accelerating, xiao2014proximal, defazio2014saga}.
The SVRG/SCSG \citep{lei2017non} adopts the technique from \citet{johnson2013accelerating} and novelly introduces a random stopping criteria for its inner loops and achieve a stochastic gradient costs of $\cO(\ep^{-3.333})$.
Very recently, two independent works, namely SPIDER \citep{fang2018spider} and SVRC \citep{zhou2018stochastic}, design sharper variance reduced gradient methods and obtain a stochastic gradient computational costs of $\cO(n^{1/2}\epsilon^{-2} \land \epsilon^{-3})$, which is state-of-the-art and \textit{near-optimal} in the sense that they achieve the algorithmic lower bound in the finite-sum setting.

\paragraph{Escaping Saddles in Single-Function Case}
Recently, many theoretical works care about convergence to an approximate second-order stationary point or escaping from saddles for the case of one single function \citep{carmon2016gradient,jin2017escape,carmon2018accelerated,carmon2017convex,agarwal2017finding,jin2018accelerated,lee2017first,du2017gradient}.
Among them, the work \citet{jin2017escape} proposed a ball-shaped-noise-perturbed variant of gradient descent which can efficiently escape saddle points and achieves a  sharp stochastic gradient computational cost of $\ep^{-2}$, which is also achieved by \NEON+GD \citep{xu2018first,allen2018neon2}.
Another line of works apply momentum acceleration techniques \citep{agarwal2017finding, carmon2017convex, jin2018accelerated} and achieve a rate of $\epsilon^{-1.75}$ for a general optimization problem.

\paragraph{Escaping Saddles in Finite-Sum Case}
For the finite-sum setting, many works have applied variance reduced gradient methods \citep{agarwal2017finding,carmon2018accelerated,fang2018spider,zhou2018finding} and further reduce the stochastic gradient computational cost to $\tO(n\ep^{-1.5} + n^{3/4}\ep^{-1.75})$ \citep{agarwal2017finding, allen2018neon2}.
\citet{reddi2018generic} proposed a simpler algorithm that obtains a stochastic gradient cost of $\tO\left(n\ep^{-1.5} + n^{3/4} \ep^{-1.75} + n^{2/3} \ep^{-2}\right)$.
With recursive gradient method applied \citep{fang2018spider,zhou2018finding}, the stochastic gradient cost further reduces to $\tO\left( (n\ep^{-1.5} + n^{3/4}\ep^{-1.75}) \land (n + n^{1/2} \ep^{-2} + \ep^{-2.5}) \right)$, which is the state-of-the-art.

\paragraph{Miscellaneous}
It is well-known that for general nonconvex optimization problem in the form of \eqref{opt_eq}, finding an approximate global minimizer is in worst-case \textit{NP-hard} \citep{hillar2013most}.
Seeing this, many works turn to study the convergence properties based on specific models.
Faster convergence rate to local or even global minimizers can be guaranteed for many statistical learning tasks such as principal component analysis \citep{li2018near, jain2017non}, matrix completion \citep{jain2013low, ge2016matrix,sun2016guaranteed}, dictionary learning \citep{sun2015nonconvex,sun2017complete} as well as linear and nonlinear neural networks \citep{zhong2017recovery,li2017convergence,li2018algorithmic}.

In retrospect, our focus in this paper is on escaping from saddles, and we refer the readers to recent inspiring works studying how to escape from local minimizers \citet{zhang2017hitting,jin2018local}.

\section{Conclusions and Future Direction}\label{sec:conclusion}
In this paper, we presented a sharp convergence analysis for the classical SGD algorithm.
We showed that equipped with a ball-controlled stopping criterion, SGD achieves a stochastic gradient computational cost of $\tO(\ep^{-3.5})$ for finding an $(\ep,\cO(\ep^{0.5}))$-approximate second-order stationary point, which improves over the best-known SGD convergence rate $\cO\left( \min(poly(d) \ep^{-8}, d\ep^{-10}) \right)$ prior to our work.
While this work focuses on sharpened convergence rate, there are still some important questions left:

\begin{enumerate}[(i)]
\item 
It is still unknown whether SGD achieves a rate that is faster than $\tO(\ep^{-3.5})$ or $\tO(\ep^{-3.5})$ is exactly the lower bound for SGD to solve the general problem in the form of \eqref{opt_eq}.
As we mentioned in \S\ref{sec:intro}, it is our conjecture that variance reduction methods are necessary to achieve an $(\ep,\sqrt{\rho\ep})$-approximate second-order stationary point in fewer than $\tO(\ep^{-3.5})$ steps.

\item 
We have \textit{not} considered several important extensions in this work, such as the convergence rate of SGD in solving constrained optimization problems, and how one extends the analysis in this paper to the proximal case.

\item
It will be also interesting to study the stochastic version of Nesterov's accelerated gradient descent (AGD) \citep{jin2018accelerated}.
\end{enumerate}

\paragraph{Acknowledgement}
The authors would like to greatly thank Chris Junchi Li providing us with a  proof of SGD to escape saddle points in $\tO(\ep^{-4})$ computational costs and  carefully revising our paper. The authors also would  like  to thank Haishan Ye for very helpful discussions and  Huan Li, Zebang Shen,  and Li Shen for very helpful comments.

Zhouchen Lin is supported by 973 Program of China (grant no. 2015CB352502), NSF of
China (grant nos. 61625301 and 61731018), Qualcomm, and Microsoft Research
Asia.

%%% Local Variables:
%%% mode: latex
%%% TeX-master: t
%%% End:

%  LocalWords:  nonconvex SGD ge online jin AGD daneshmand xu allen
%  LocalWords:  eq NESTEROV tripuraneni nesterov agarwal VR natasha
%  LocalWords:  carmon zhou SG RSGD SNVRG SFO schmidt johnson xiao du
%  LocalWords:  defazio SVRG SCSG novelly SVRC reddi hillar li jain
%  LocalWords:  minimizers zhong zhang Nesterov's

\pb
\bibliographystyle{apalike2}
\bibliography{Nonconvex}

\pagebreak\appendix

\pb\section{Proof Sketches for Theorem \ref{theo:main}}
\label{sec:proof_tech}
 We briefly introduce our proof techniques to prove our main Theorem \ref{theo:main} in this section. The rigorous proof is shown in Appendix \ref{proofone}, \ref{prooftwo}, and \ref{proofthree}.
For convenience, when  we study Algorithm \ref{algo:SGD} in each inner loop from Line \ref{line3} to Line \ref{line10},  we  override the definition of $\x^0$ as its initial vector.   Our poof  basically  consists of two ingredients. The first is to prove that SGD can efficiently escape saddles: with high probability, if $\lambda_{\min} \left(\nabla^2 f(\x^{0}) \right)  \leq -\delta_2\asymp -\ep^{0.5}$,  $\x^k$ moves out of $\cB(\x^0, B)$ in $K_0$ iterations (Refer to Appendix \ref{proofone}).
The second is to show that SGD  converges with a faster rate of $\tO(\ep^{-3.5})$, rather than $\tO(\ep^{-4})$. We further separate the second destination into two parts:
\begin{enumerate}
	\item Throughout the execution of the algorithm,  each time $\x^k$ moves out of $\cB(\x^0, B)$,  with high probability, the function value shall decrease with a magnitude at least $\tO(\ep^{1.5})$ (Refer to Appendix \ref{prooftwo}). 
	\item Once  $\x^k$  does not move out of $\cB(\x^0, B)$ until $K_0$ iteration, with high probability, we  find a desired approximate second-order stationary point (Refer to Appendix \ref{proofthree}).
\end{enumerate}

Let $\cF^k = \sigma \{\x^0, \bzeta^1, \cdots, \bzeta^k  \}$ be the filtration involving the full information of all the previous $k$ times
iterations, where $\sigma\{\cdot\}$ denotes the sigma field. And let $\sK_0$ be the first time (mathematically, a stopping time) that $\x^k$ exits the $\radius$-neighborhood of $\x^0$, i.e.	
\beq\label{sK0}
\sK_0  = \inf_k\{k\geq 0: \|\x^{k} - \x^{0} \| > \radius \}
.
\eeq
Both  $\x^{k}$ and $\cI_{\sK_0>k}$ is measurable on $\cF^{k}$, where  $\cI$ denotes the indicator function.

\subsection{Part \mbox{I}: Escaping Saddles}\label{1Escaping Saddles} 

Our goal is to prove the following proposition:

\begin{proposition}\label{prop:first}
	Assume $\lambda_{\min} \left(\nabla^2 f(\x^{0}) \right)  \leq -\delta_2$, and recall the parameter set in \eqref{parametersetting}. 
	Initialized at $\x^0$ and running Line \ref{line3} to Line \ref{line10}, with probability at least $1- \frac{p}{3}$ we have 
	\beq\label{sK0bdd}
	\sK_0
	\le K_0 = 
	\left(
	\left\lfloor \frac{\log( 3 \cdot p^{-1})}{ \log (0.7^{-1}) } \right\rfloor + 1
	\right) K_o
	,
	\eeq
	where $K_o = 2 \log\left( \frac{24\sqrt{d}}{\eta}\right)\eta^{-1}\delta_2^{-1}$.
\end{proposition}

Proposition \ref{prop:first} essentially says that assuming if the function has a negative Hessian eigenvalue $\le -\epH$ at $\x^0$, the iteration exits the $\radius$-neighborhood of $\x^0$ in $K_0 = \tilde\cO(\eta^{-1} \epH^{-1})$ steps with a high probability.

To prove Proposition \ref{prop:first}, we let $\w^k(\u)$, $k\geq 0$ be the iteration by SGD starting from a fixed $\u \in \RR^d$ using the same stochastic samples as iteration $\x^k$, i.e.
\beq\label{wku}
\w^k(\u) = \w^{k-1}(\u) - \eta \nabla F( \w^{k-1}(\u); \bzeta^{k})
.
\eeq
Obviously, we have $\x^k = \w^k(\x^0)$.
Let $\cK_{exit}(\u)$ be the first step number $k$ (a stopping time) such that $\w^k(\u)$ exits the $\radius$-neighborhood of $\x^0$.
Formally,
\beq\label{cKexit}
\cK_{exit}(\u) := \inf\{k\ge 0: \|\w^{k}(\u) - \x^0 \| > \radius\}
.
\eeq
It is easy to see from \eqref{sK0} that $\sK_0 = \cK_{exit}(\x^0)$.
Inspired from \cite{jin2017escape}, we cope with the stochasticity of gradients and define the so-called \textit{bad initialization region} as the point $\u$ initialized from which iteration $\w^k(\u)$ exits the $\radius$-neighborhood of $\x^0$ with probability $\le 0.4$:
\beq\label{bad_initS}
\cS_{\Ko}^{\radius}(\x^0)
:=
\left\{\u \in \RR^d: \PP\left(\cK_{exit}(\u)  < \Ko \right) \le 0.4 \right\}
.
\eeq
We will show that the bad initialization region $\cS_{\Ko}^{\radius}(\x^0)$ enjoys the $(q_0,\e_1)$-narrow property, where $q_0 = \frac{\sigma}{4\sqrt{d}}$.
Since the first step will provide a continuous noise as supposed by Assumption \ref{assu:noise},  with the properly selected $q_0$,  it will move the iteration out of the bad initialization region in its first step with probability $\ge 3/4$.
Repeating such an argument in a logarithmic number of rounds enables escaping to occur with high probability.

The idea is to prove the following lemma:

\begin{lemma}\label{lemm:key}
	Let the assumptions of Proposition \ref{prop:first} hold, and assume WLOG $\e_1$ be an arbitrary eigenvector of $\nabla^2 f(\x^0)$ corresponding to its smallest eigenvalue $-\delta_m$, which satisfies $\delta_m \ge \delta_2 > 0$.
	Then we have  for any fixed $q \ge q_0$ and pair of points $\u,\u+q\e_1 \in \cB(\x^0,\radius)$ that
	\beq\label{kep_rep}
	\begin{split}
		&
		\PP\left( \cK_{exit}(\u) \ge K_o
		\text{ and } \cK_{exit}(\u+q\e_1) \ge K_o \right)
		\le
		0.1
		.
	\end{split}\eeq
\end{lemma}

Lemma \ref{lemm:key} is inspired from Lemma 15 in \cite{jin2017escape}.  Nevertheless due to the noise brought in at each update step, the analysis of stochastic gradient differs from that of the gradient descent in many  aspects. For example,  instead of showing the decrease of function value, we need to show that  with a positive probability, at least one of the two iterations, $\w^k(\u+q\e_1)$ or $\w^k(\u)$, exits the $\radius$-neighborhood of $\x^0$.  Our proof is also more intuitive compared with  Lemma 15 in \cite{jin2017escape}. The core idea is to focus on   analyzing the  difference trajectory for $\w^k(\u+q\e_1)$ and $\w^k(\u)$, and  to show that the rotation speed for the difference trajectory  is the same as the expansion speed.  Detailed proof is provided in \S\ref{sec:proof,lemm:key}.

\subsection{Part \mbox{II}: Faster Descent}\label{2Sufficient Descent}

The goal of Part \mbox{II} is to prove the following proposition:
\begin{proposition}[Faster Descent]\label{faster descent}
	For  Algorithm \ref{algo:SGD} with parameter set in \eqref{parametersetting}. With probability at least $1-\frac{2}{3}p$,  if $\x^k$ moves out of  $\cB(\x^{0}, B)$ in $K_0$ iteration, we have
	\begin{eqnarray}
	f\left(\x^{\sK_0}\right) \leq f\left(\x^{0}\right) - \frac{B^2}{7\eta K_0}.
	\end{eqnarray}	
\end{proposition}

Proposition \ref{faster descent} is the key for  SGD to achieve the reduced $\tO(\ep^{-3.5})$ stochastic computation costs.
It shows that no matter what does the local surface of $f(\x)$ look like, once $\x^k$ moves out of the ball in $K_0\asymp \ep^{-2}$ iterations, the function value shall decrease with a magnitude of at least $\tilde{O}(\ep^{1.5})$.  To put it differently,   on average, the function value  decreases  at least $\tO(\epsilon^{3.5})$ per-iteration during the execution of Algorithm \ref{algo:SGD}. We will present the basic argument below.

We start with reviewing the more traditional approach for proving sufficient descent of SGD, and then we will discuss how to improve it as done in this work.  
The previous approaches are all based on the idea of \citep{NESTEROV}, which   mainly takes advantage of  the  gradient-smoothness condition of the objective.  The proof can be briefly described below:
\begin{eqnarray}
\E_{\bzeta}  f(\x^{k+1}) &\leq& f(\x^k) + \E_{\bzeta} \<\nabla f(\x^k), \x^{k+1} - \x^{k} \> + \frac{L}{2}\E_{\bzeta}\| \x^{k+1} - \x^k\|^2 \notag\\
&\overset{\eqref{SGD}}{=}& f(\x^k)  - \left(\eta - \frac{L\eta^2}{2}\right) \| \nabla f(\x^k)\|^2 + \frac{\eta^2L}{2}\E_{\bzeta}\| \nabla F(\xb^{k}; \bzeta^{k+1}) - \nabla f(\x^k)\|^2\label{tra1}\notag\\
&\overset{\text{Assum.} \ref{assu:noise}}{\leq}& f(\x^k)  - \left(\eta - \frac{L\eta^2}{2}\right) \| \nabla f(\x^k)\|^2 + \frac{\eta^2L \sigma^2}{2}. \label{tra}
\end{eqnarray}
From the above derivation, in order to guarantee the monotone descent of function value in expectation, the step size $\eta$ needs to be
\begin{eqnarray}\label{step size}
\eta = \cO \left(\frac{ \| \nabla f(\x^k)\|^2}{L \sigma^2}\right) = \cO(\epsilon^{2}),
\end{eqnarray}
where the last equality uses $\left\| \nabla f(\x^k)\right\|\geq \ep$.
Plugging  \eqref{step size} into \eqref{tra}, and using $\| \nabla f(\x^k)\|\geq \ep$, we have that the function value per-iteration would descent with a magnitude of at least $\cO(\epsilon^{4})$. Such result indicates that, in the worse case, SGD  takes $\cO(\epsilon^{-4})$ stochastic oracles to find an $\ep$-approximate first-order stationary point. This simple argument is the reason why previous works conjectured 
that the complexity of SGD is $\cO(\epsilon^{-4})$.

However, in this paper, we show that the above analysis can be further improved by using the Hessian-smoothness condition of the objective, and by considering the
decomposition of objective function $f(\x)=f_+(\x)+ f_-(\x)$, and treating component $f_+(\x)$ and component $f_-(\x)$ separately as follows:
\begin{itemize}
	\item (Case 1) The component $f_+(\x)$ is near convex locally, in the sense that
	$\lambda_{\min}\left(\nabla^2 f(\x)\right) \geq -  \Omega(\ep^{0.5})$  for all $\x\in\cB(\x^0, B)$. In this case, by using techniques for near convex problems, it is possible for us to take a larger stepsize $\eta=\cO(\epsilon^{1.5})$ and prove  a faster convergence rate.
	\item (Case 2) The component $f_-(x)$ is near concave locally,  in the sense that $\lambda_{\max}\left(\nabla^2 f(\x)\right) \leq   \cO(\ep^{0.5})$ for all $\x\in\cB(\x^0, B)$ .  In this case, 
	It can be shown that the last term on the right hand side  of \eqref{tra} can be reduced to $\cO(\eta^2\ep^{0.5}\sigma^2)$. Therefore
	the step size can be chosen as $\eta=\cO(\ep^{1.5})$, leading to a fast function value reduction.
\end{itemize}
To formalize the above  observations into a rigorous proof, in this paper we introduce the quadratic approximation of $f(\x)$ at point $\x^0$, defined as
\begin{eqnarray}
g(\x) \coloneqq  \left[\nabla f(\x^{0})\right]^\bT\left(\x - \x^{0}\right) +  \frac{1}{2}\left[\x - \x^{0}\right]^\bT \nabla^2 f(\x^{0}) \left[\x - \x^{0}\right].
\end{eqnarray}  
We let   $\cS$ be the subspace spanned by all eigenvectors of $\nabla^2 f(\x^{0})$ whose eigenvalue is greater than $0$, and $\cS\bot$ denotes the complement space.  Also let $\cP_{\cS} \in \RR^{d\times d }$ and $\cP_{\cS\bot}\in \RR^{d\times d }$ as the projection matrices onto the space of $\cS$ and $\cS\bot$, respectively.  Also  let  the full SVD decomposition of $\nabla^2 f(\x^{0})$ be $\mathbf{V}\sum \mathbf{V}^T$. We introduce $\bH_{\cS} = \mathbf{V}\sum_{(\lambda_i >0)}\mathbf{V}^T$ and  $\bH_{\cS\bot} = \mathbf{V}\sum_{(\lambda_i \leq0)}\mathbf{V}^T$ respectively,  and define the following two auxiliary functions $g_{\cS}:\cS\to \RR$ and  $g_{\cS\bot}:\cS\bot\to \RR$:
\begin{eqnarray}
g_{\cS}(\uu)\coloneqq   \left[\cP_{\cS}\nabla f\left(\x^{0}\right)\right]^\bT\uu +  \frac{1}{2}\uu^T \bH_{\cS} \uu,
\end{eqnarray}
and 
\begin{eqnarray}
g_{\cS\bot}(\vv) \coloneqq   \left[\cP_{\cS\bot}\nabla f\left(\x^{0}\right)\right]^\bT\vv +  \frac{1}{2}\vv^T \bH_{\cS\bot} \vv .
\end{eqnarray}
For the previously mentioned decomposition of $f(\x)=f_+(\x)+f_-(\x)$, one may simply take
$f_+(\x)=f(\cP_{\cS} x)$, and let $f_-(\x)= f(\x)-f_+(\x)$. It can be checked that
$f_+(\cdot)=f_+(\x^0)+ g_{\cS}(\cdot) + \tO(\epsilon^{1.5})$ and $f_-(\cdot)=f_-(\x^0)+g_{\cS\bot}(\cdot)+\tO(\epsilon^{1.5})$. It follows that we only need to separately
analyze the two quadratic approximations $g_{\cS}(\cdot)$ and $g_{\cS\bot}(\cdot)$.
We then bound the difference between $f(\x^{\sK_0})$ and $g_{\cS}(\x^{\sK_0} - \x^0) + g_{\cS\bot}(\x^{\sK_0}-\x^0)+f(\x^0)$ as $\tO(\epsilon^{1.5})$.

The analysis for  $g_{\cS\bot}(\cdot)$  can be obtained via the standard analysis  informally described above in Case~2 (Refer to Lemma \ref{analysis on vv}).  

Our proof technique for dealing with $g_{\cS}(\cdot)$ is to introduce an auxiliary trajectory with the following deterministic updates for $k= 0,1, 2, \dots,$ as:
\begin{eqnarray}
\y^{k+1} = \y^k - \eta \nabla g_{\cS}\left(\y^k\right),
\end{eqnarray}
and $\y^{0} = \mathbf{0}$.  We then track and analyze the difference trajectory between $\cP_{\cS} \left(\x^{\sK_0} - \x^{0}\right)$ and $\y^{\sK_0}$ (Refer to Lemma \ref{zz}).  In the sense that $\y^k$ simply performs Gradient Descent, we can arrive our final results for $g_{\cS}(\cdot)$ (Refer to Lemma \ref{analysis on uu}), which leads to a rigorous statement of Case~1.

Finally, via the fact that $\x^k$ moves out of the ball in $K_0$ iteration throughout the execution of Algorithm \ref{algo:SGD},  we prove that with high probability  the sum for the norm of gradients can be  lower bounded as:
\begin{eqnarray}\label{2endim2}
&& \sum_{k = 0}^{\sK_0-1}  \left\| \nabla g_{\cS\bot}\left(\x^{k} - \x^0\right)  \right\|^2+\sum_{k=0}^{\sK_0-1}\left\|  \nabla  g_{\cS} \left(\y^k\right)\right\|^2 =  \tilde{\Omega}(1),
\end{eqnarray}
which ensures sufficient descent of  the function value. By putting the above arguments
together, we can obtain Proposition \ref{faster descent}.

\subsection{Part \mbox{III}: Finding SSP} 
Part \mbox{III}  proves the following proposition:
\begin{proposition}\label{local convergence}
	With probability of at least $1-p$, if $\x^k$ has not moved out of the ball in $K_0$ iterations, then let $\bar{\x} =\sum_{k=0}^{K_0-1}\x^k$,  we have
	\begin{eqnarray}
	\|\nabla f(\bar{\x})\| \le  18 \rho B^2\asymp \ep
	,\qquad
	\lambda_{\min}(\nabla^2 f(\bar{\x})) \ge - 17\delta \asymp  -\sqrt{\rho\ep}
	.
	\end{eqnarray}
\end{proposition}
Proposition \ref{local convergence} can be obtained via  the same idea of  Part \mbox{II}. We first study the quadratic approximation function $g(\bar{\x})$ and then bound the difference between $g(\bar{\x})$ and $f(\bar{\x})$. 

Finally, integrating Proposition \ref{prop:first},  \ref{faster descent}, and \ref{local convergence}, and using the boundedness of the function value  in Assumption \ref{assu:main}, we know with probability at least $1- (T_1+1) p$, Algorithm \ref{algo:SGD} shall stop before $T_0$ steps, and output an approximate second-order stationary point satisfying \eqref{ESSP}, which immediately leads to Theorem \ref{theo:main}.

%  LocalWords:  SGD parametersetting sK stochasticity WLOG NESTEROV
%  LocalWords:  eigenvector Assum tra eigenvectors SVD SSP ESSP

\section{ Concentration Inequalities}\label{sec:pinelis}
In our proofs, concentration inequalities are fundamental to obtain the high-probability result.  Before we prove our results, we introduce the following two (advanced) inequalities which  will be used in our proofs.

\subsection{Vector-Valued Concentration Inequality}

\begin{theorem}[Vector-Martingale Azuma–Hoeffding, Theorem 3.5 in \citet{pinelis1994optimum}]\label{proh}
	Let $\bm\ep_{1:K} \in \RR^d$ be a vector-valued martingale difference sequence with respect to $\cF^k$, i.e. for each $k=1,\dots,K$, $\EE[\bm\ep_k \mid \cF^{k-1}] = 0$ and $\|\ep_k\|^2 \leq B_k^2$. We have 
	\beq\label{azumain}
	\PP \left(
	\left\|
	\sum_{k=1}^{K} \bm\ep_k
	\right\|
	\ge \lambda \right)
	\le
	4 \exp \left(
	-\frac{\lambda^2}{4 \sum_{k=1}^{K} B_k^2} 
	\right),
	\eeq
	where $\lambda$ is an arbitrary real positive number.
\end{theorem}

Theorem \ref{proh} is not a straightforward derivation of one-dimensional Azuma's inequality. Because  the bound on the right hand of \eqref{azumain} is \textit{dimension-free}. Such result might be first found by \citet{pinelis1994optimum}. See also \citet{kallenberg1991some}, Lemma 4.4 in \citet{zhang2005learning} or Theorem 2.1 in \cite{zhang2005learning} and the references therein.

\subsection{Data-Dependent Concentration Inequality}
\begin{theorem}[Date-Dependent Concentration Inequality, Lemma $3$ in \citet{rakhlin2012making}]\label{proh2}
	Let $\ep_{1:K} \in \RR$ be a  martingale difference sequence with respect to $\cF^k$, i.e. for each $k=1,\dots,K$, $\EE[\bm\ep_k \mid \cF^{k-1}] = 0$, and $$\EE[\ep_k^2\mid \cF^{k-1}] \leq \sigma^2_k.$$
	Furthermore, assume that $\PP(\|\ep_k\|\leq b\mid \cF^{k-1}) = 1$.  Let $V_K^2 = \sum_{k=1}^K \sigma_k^2$,  for any $\delta<1/e$ and $K\geq 4$, we have 
	\begin{eqnarray}
	\PP\left(\sum_{k=1}^K \ep_k > 2 \max\left\{2\sqrt{V_k}, b \sqrt{\log(1/\delta)}\right\}\sqrt{\log(1/\delta)}   \right) \leq \log(K)\delta. 
	\end{eqnarray}	
	
\end{theorem}
Theorem \ref{proh2} extends the standard Freedman's Inequality \citep{Freedman1975On} by allowing  $\sigma_k$ being the conditional
variance.  Similar results can be found  in \cite{bartlett2008high-probability} and  Lemma $2$ in \citet{zhang2005learning}  and the references therein.

Note that Theorem \ref{proh} and \ref{proh2}  only list the results for the bounded   martingale difference. Similar results can also be established when the martingale difference follows from a  sub-gaussian distribution. In the rest of our proofs, we also only present the results for the bounded  noise case, i.e. \eqref{bounded} in Assumption \ref{assu:noise}. Analogous analysis can be applied for sub-gaussian noise, i.e. \eqref{bound2} in Assumption \ref{assu:noise}.

\section{Deferred Proofs of Part \mbox{I}: Escaping Saddles}\label{proofone}
Let the deterministic time 
\beq\label{Todef}
\Ko
 = 2 \log\left( \frac{24\sqrt{d}}{\eta}\right)\eta^{-1}\delta_2^{-1} \geq
\left\lceil
\frac{\log (6 / q_0) }{\log (1 + \eta (\delta_2) )}
\right\rceil\overset{B\leq1}{\geq} \left\lceil
\frac{\log (6\radius / q_0) }{\log (1 + \eta (\delta_2) )}
\right\rceil
,
\eeq
where  $q_0 = \frac{\sigma\eta}{4\sqrt{d}}$.
We  prove Proposition \ref{prop:first} that bound the iteration number  to escape $\cB(\x^0, \radius)$.

\begin{proof}[Proof of Proposition \ref{prop:first}]
\begin{enumerate}[(i)]
\item
We prove in this item that $\cS_{\Ko}^{\radius}(\x^0)$ satisfies the $(q_0,\e_1)$-narrow property, i.e.~there cannot be two points $\u,\u+q\e_1 \in \cS_{\Ko}^{\radius}(\x^0)$ such that $q \ge q_0$.
Indeed if such two points do exist, from \eqref{bad_initS} we have
$$
 \PP( \cK_{exit}(\u) \ge \Ko) \ge 0.6
\quad\text{and}\quad
 \PP( \cK_{exit}(\u+q\e_1) \ge \Ko) \ge 0.6
,
$$
and hence by inclusion-exclusion principle
\begin{align*}
&\quad
\PP( \cK_{exit}(\u) \ge \Ko \text{ and } \cK_{exit}(\u+q\e_1) \ge \Ko)
 \\&\ge
\PP( \cK_{exit}(\u) \ge \Ko) + \PP( \cK_{exit}(\u+q\e_1) \ge \Ko)  -  1
 \ge
2(0.6) - 1 = 0.2
,
\end{align*}
which contradicts \eqref{kep_rep} in Lemma \ref{lemm:key}.

\item
Combining the fact that $\cS_{\Ko}^{\radius}(\x^0)$ satisfies the $(q_0,\e_1)$-narrow property (as in Definition \ref{defi:qnarrow}) where $q_0 = \eta \sigma / 4\sqrt{d}$, and Assumption \ref{assu:noise} which allows $\nabla F(\u; \bzeta^1)$ to satisfy the $\e_1$-disperse property, we have for any $\u\in \RR^d$ the following holds:
\beq\label{smallprob}
\begin{split}
\PP\left( \w^1(\u)  \in \cS_{\Ko}^{\radius}(\x^0)\right)
&=
\PP\left( \u - \eta \nabla F( \u; \bzeta^1) \in \cS_{\Ko}^{\radius}(\x^0)\right)
 \\&=
\PP\left( \nabla F(\u; \bzeta^1) \in \eta^{-1} [ - \cS_{\Ko}^{\radius}(\x^0) + \u ] \right)
 \le \frac14
,
\end{split}
\eeq
where we applied \eqref{wku} and that $\w^0(\u) = \u$.
Thus
\begin{align*}
\PP\left( \cK_{exit}(\u) \le \Ko  \right) 
 &\ge
\EE\left( \PP\left( \cK_{exit}(\w^1(\u)) < \Ko  \mid  \cF^1\right);
 \w^1(\u) \in \cS_{\Ko}^{\radius}(\x^0)  \right)
 \\&\quad\quad
 +
\EE\left( \PP\left( \cK_{exit}(\w^1(\u)) < \Ko  \mid  \cF^1\right);
 \w^1(\u) \in  [\cS_{\Ko}^{\radius}(\x^0)]^c \right)
 \\&\ge
\EE\left( \PP\left( \cK_{exit}(\w^1(\u)) < \Ko  \mid  \cF^1\right); 
\w^1(\u) \in [\cS_{\Ko}^{\radius}(\x^0)]^c \right)
 \\&\ge
  0.4 \PP\left( \w^1(\u) \in [\cS_{\Ko}^{\radius}(\x^0)]^c \right)
 \ge
 0.4 \left( \frac34 \right)
 =
  0.3
 ,
\end{align*}
i.e.~$\sup_{\u'\in\RR^d} \PP\left( \cK_{exit}(\u') > \Ko  \right) \le 0.7$.
Using \eqref{smallprob} and Markov's property we conclude for any $N \ge 1$
\begin{align*}
\PP\left( \cK_{exit}(\u) > N \Ko  \right) 
 &=
\EE\left( \PP(\cK_{exit}(\w^{(N-1)\Ko}(\u)) > \Ko \mid \cF^{K_0});  \cK_{exit}(\u) > (N-1) \Ko  \right)
 \\&\le
\sup_{\u'\in \cB(\u, \radius)} \PP\left( \cK_{exit}(\u') > \Ko  \right)
  \cdot
\PP( \cK_{exit}(\u) > (N-1) \Ko  )
 \\&\le
0.7\cdot \PP( \cK_{exit}(\u) > (N-1) \Ko  )
,
\end{align*}
which further leads to
$
\PP\left( \cK_{exit}(\u) > N \Ko  \right) 
 \le
0.7^N
$.
Letting $N = \lfloor \log(3\cdot p^{-1}) / \log (0.7^{-1}) \rfloor + 1$ we obtain an exit probability of $\le p/3$ which completes the proof of Proposition \ref{prop:first}.

\end{enumerate}
\end{proof}

\subsection{Proof of Lemma \ref{lemm:key}}\label{sec:proof,lemm:key}
This subsection denotes to the proof of Lemma \ref{lemm:key} in the following steps:

\begin{enumerate}[(i)]
\item
Denote for simplicity $\w^k \equiv \w^k(\u)$, and $\bar{\w}^k \equiv \w^k(\u+q\e_1)$.
Recall from the SGD update rule we have $\w^0 = \u$, and for all $k=1,2,\dots$ for a random index $\bzeta^{k}$ drawn from distribution $\cD$,
$$
\w^k = \w^{k-1} - \eta \nabla F(\w^{k-1}; \bzeta^{k})
,
$$
and
$$
\bar{\w}^k = \bar{\w}^{k-1} - \eta \nabla F(\bar{\w}^{k-1}; \bzeta^{k})
.
$$
Recall the definition of $\cK_{exit}(\u)$ in \eqref{cKexit}, we let
\beq\label{cT1}
\cK_1 :=  
\cK_{exit}(\u) \land \cK_{exit}(\u + q\e_1)
.
\eeq
For our analysis, we define a coupled $\cF^k$-measurable iteration $\z^k$, as follows:
\beq\label{yk}
\z^k
=
\left\{
\begin{array}{ll}
 \bar{\w}^k - \w^k 						&   \text{on $(k < \cK_1)$}
\\
 \left( \Ib - \eta \nabla^2 f(\x^0)  \right)  \z^{k-1}  	&   \text{on $(k\ge \cK_1)$}
\end{array}
\right.
,
\eeq
i.e. we couple the difference iteration $\bar{\w}^k - \w^k$ on $(k< \cK_1)$, and keep moving the iteration afterwards as if it is the difference iteration of SGD for pure quadratics (we eliminate both the Taylor remainder term and noise term after exiting).
Since $\w^0 = \u$, $\bar{\w}^0 = \u + q \e_1$ and $(\cK_1 > 0)$ holds, we have $\y^0 = q \e_1$.
We \textit{only} want to show for any $\u$ such that $\u,\u+q\e_1 \in \cB(\tilde{\x},\radius)$,
\beq\label{cT1bdd}
\PP( \cK_1 > \Ko )
=
\PP\left(  \cK_{exit}(\u) \land \cK_{exit}(\u + q\e_1)  > \Ko \right)
\le p
.
\eeq

\item
Letting $\Hb = \nabla^2 f(\x^0)$, and we first conclude the following lemma to express $\z^k$ defined in \eqref{yk}:

\begin{lemma}\label{lemm:ytexpress}
We have for all $k=1,2,\dots$
\beq\label{ytexpress}
\z^k
=
\left( \Ib - \eta \Hb  \right)  \z^{k-1}
+
\eta \Db^{k-1} \z^{k-1}
+ 
\eta \bxi_d^{k}
,
\eeq
where
\beq\label{eqD}
\| \Db^{k-1} \|
 \le 
\rho \max\left( \|\bar\w^{k-1} - \x^0\|, \|\w^{k-1} - \x^0\| \right)
 \le 
\rho \radius
,
\eeq
$\{\bxi_d^{k}\}$  forms a martingale difference sequence  satisfying
\beq\label{eqep}
\|\bxi_d^{k} \|
 \le 
2L \|\z^{k-1}\|
.
\eeq
\end{lemma}

\begin{proof}[Proof of Lemma \ref{lemm:ytexpress}]
By setting $\Db^{k-1} = \mathbf{0}^{d\times d}$ and $\bxi_d^{k} = \mathbf{0}$, on event $(k \ge \cK_1)$ we can easily see from \eqref{yk} that all \eqref{ytexpress}, \eqref{eqD} and \eqref{eqep} hold, since their left hands are zero.
For its complement $(k < \cK_1)$, we have
\begin{align*}
\z^k = \bar{\w}^{k} - \w^{k}
  &=
\bar{\w}^{k-1} - \w^{k-1} - \eta\left(
  \nabla F(\bar{\w}^{k-1}; \bzeta^{k}) - \nabla F(\w^{k-1}; \bzeta^{k})
\right)
  \\ &=
\z^{k-1} - \eta\left(
	\nabla f(\bar{\w}^{k-1})
	-
	\nabla f(\w^{k-1})
	\right)
  \\&\quad+
	\eta\left[
	\left(\nabla f(\bar{\w}^{k-1}) - \nabla f(\w^{k-1}) \right)
	 - 
	\left(\nabla  F(\bar{\w}^{k-1}; \bzeta^{k}) - \nabla F(\w^{k-1}; \bzeta^{k}) \right)
	\right]
 \\ &= 
\z^{k-1} - \eta
\left[
\int_0^1 \nabla^2 f\left(\w^{k-1} + \theta(\bar{\w}^{k-1} - \w^{k-1}) \right)  \ud \theta
\right]
\z^{k-1}
+
\eta \bxi_d^{k}
 \\ &\equiv
\z^{k-1} - \eta\left( \Hb - \Db^{k-1} \right) \z^{k-1}
+ \eta \bxi_d^{k}
,
\end{align*}
where we set the following terms \eqref{Dbterm} and \eqref{xiterm}:
\beq\label{Dbterm}
\Db^{k-1}
 \equiv 
\nabla^2 f(\x^0)
-\int_0^1 \nabla^2 f\left(\w^{k-1} + \theta(\bar{\w}^{k-1} - \w^{k-1}) \right) d\theta
,
\eeq
and the noise term $\bxi_d^{k}$ generated at each iteration
\beq\label{xiterm}
\bxi_d^{k}
 \equiv 
	\left(\nabla f(\bar{\w}^{k-1}) - \nabla f(\w^{k-1}) \right)
	 - 
	\left(\nabla  F(\bar{\w}^{k-1}; \bzeta^{k}) - \nabla F(\w^{k-1}; \bzeta^{k}) \right)
,
\eeq
proving \eqref{ytexpress}.

It leaves us to prove \eqref{eqD} and \eqref{eqep}.
From \eqref{Dbterm}, we have
\begin{align*}
\| \Db^{k-1} \|
 &\le 
\int_0^1 
\left\|  
\nabla^2 f(\x^0)
-
\nabla^2 f\left(\w^{k-1} + \theta(\bar{\w}^{k-1} - \w^{k-1}) \right)
\right\|
 d\theta
 \\&\le 
\rho \int_0^1 
\left\|  
\theta (\bar{\w}^{k-1}-\x^0) + (1-\theta) (\w^{k-1}-\x^0)
\right\|
 d\theta
 \\&\le 
\rho\max\left( \|\bar{\w}^{k-1} - \x^0\|, \|\w^{k-1} - \x^0\| \right)
,
\end{align*}
which is bounded by $\rho\radius$ since $\max\left( \|\bar{\w}^{k-1} - \x^0\|, \|\w^{k-1} - \x^0\| \right) \le \radius$, proving \eqref{eqD}.

The $\bxi_d^k$ defined in \eqref{xiterm} has $\EE[\bxi_d^k \mid \cF^{k-1}] = 0$ forming a Martingale Difference Sequence, and from Lipschitz continuity of the objective function, we have
\begin{align*}
\|\bxi_d^{k} \|
 &\le 
\left\|
	\nabla f(\bar{\w}^{k-1}) - \nabla f(\w^{k-1}) 
\right\|
	 +
\left\|
	\nabla  F(\bar{\w}^{k-1}; \bzeta^{k}) - \nabla F(\w^{k-1}; \bzeta^{k}) 
\right\|
 \\ &\le 
L \left\| \bar{\w}^{k-1} - \w^{k-1} \right\|
+
L \left\| \bar{\w}^{k-1} - \w^{k-1} \right\|
 =
2L \|\z^{k-1}\|
.
\end{align*}
This completes the proof of \eqref{eqep}, and hence the lemma.

\end{proof}

\item
We observe from \eqref{ytexpress} that if $\nabla^2 f(\z)$ does \textit{not} rotate in the sense that each pair of Hessian matrices $\nabla^2 f(\w_1)$ and $\nabla^2 f(\w_2)$ can be spectrally decomposed via the same orthogonal matrix, one can analyze the iteration coordinate-wisely.
Here, the rotation effect of Hessian matrix \textit{cannot} be ignored.
Hence, we analyze the difference iteration $\z^k$ in two aspects:
(i) $\z^k$ has a \textit{rotation effect} after standardization, and 
(ii) its norm $\|\z^k\|$ has an \textit{expansion effect}.

To decouple these two effect, we define a rescaled iteration as follows.
Let $\delta_m$ denote the negated least eigenvalue $\lambda_{\min}(\nabla^2 f(\x^0))$ of Hessian so $\delta_m \geq \epH$.
Let for each $k=0,1,\dots$
\beq\label{eqpsi}
\bpsi^k \equiv q^{-1} (1 + \eta \delta_m)^{-k} \z^k
.
\eeq
We  state the following lemma for the update rule of $\bpsi^k$.

\begin{lemma}\label{lemm:psitexpress}
Let
$
\hat\Db^k \equiv (1 + \eta\delta_m)^{-1} \Db^k
,
$
and
$
\bzeta_d^k \equiv q^{-1} (1 + \eta \delta_m)^{-k} \bxi_d^{k}
.
$
We have $\bpsi^0 = \e_1$ and
\beq\label{psitexpress}
\bpsi^k
=
\frac{(\Ib - \eta \Hb)}{1 + \eta \delta_m} \bpsi^{k-1}
+
\eta  \hat\Db^{k-1} \bpsi^{k-1}
+ 
\eta  \bzeta_d^{k}
,
\eeq
where
\beq\label{eqhatD}
\| \hat{\Db}^{k-1} \|
 \le 
\rho \radius
,
\eeq
and the rescaled noise iteration $\bzeta_d^k$ has
\beq\label{hatepk}
\|\bzeta_d^k \|
 \le
2L \| \bpsi^{k-1}  \|
, \quad k \geq 1. 
\eeq
Then with the step size set in \eqref{parametersetting}\footnote{We actually only  need $\eta \leq \tO(\ep^{0.5})$ to obtain  Lemma \ref{lemm:psitexpress}. },  we have on the event $\cH_o$ (\eqref{probability-3} happens),  the norm of $\bpsi^k$ satisfies
\beq\label{psinorm_express}
\| \bpsi^k \|^2
\le
4
,
\eeq
and for the projection of $\bpsi^k$ onto the first coordinate,
\beq\label{psi1_express}
\e_1^\top \bpsi^k
>
\frac12
.
\eeq
\end{lemma}

\begin{proof}[Proof of Lemma \ref{lemm:psitexpress}]
We have from the definition of $\bzeta_d^k$
\begin{align*}
\|\bzeta_d^k \|
 &\le 
q^{-1} (1 + \eta \delta_m)^{-k} \|\bxi_d^{k} \|
 \\&\le
2L q^{-1} \frac{(1 + \eta \delta_m)^{-(k-1)}}{1+\eta \delta_m}  \|\z^{k-1}\|
 \le
2L \| \bpsi^{k-1}  \|
,
\end{align*}
establishing \eqref{hatepk}, and hence
\begin{align*}
\bpsi^k
&=
q^{-1} (1 + \eta \delta_m)^{-k} \z^k
\\&=
\frac{(\Ib - \eta \Hb)}{1 + \eta \delta_m} q^{-1} ( 1 + \eta \delta_m )^{-(k-1)}  \z^{k-1}
\\&\hspace{1in}+
\eta \frac{\Db^{k-1}}{1 + \eta \delta_m} q^{-1}  (1 + \eta \delta_m)^{-(k-1)} \z^{k-1}
+ 
\eta q^{-1} (1 + \eta \delta_m)^{-k} \bxi_d^k
\\&=
\frac{(\Ib - \eta \Hb)}{1 + \eta \delta_m} \bpsi^{k-1}
+
\eta  \hat\Db^{k-1} \bpsi^{k-1}
+ 
\eta  \bzeta_d^k
,
\end{align*}
proving \eqref{psitexpress} and \eqref{eqhatD}.

To handle the term involving the $\bzeta_d^{k}$ terms on the right hands of \eqref{psi1_express} and \eqref{psinorm_express}, we first set
\beq\label{hatpsi}
\hat\bpsi^{k-1}
 =
\frac{\left[ \Ib - \eta \Hb  \right]}{1 + \eta \delta_m} \bpsi^{k-1}
.
\eeq
Since $\eta L \le 1$ we simply have $ \left[ \Ib - \eta \Hb  \right]$ is symmetric and has all eigenvalues in $[0,1+\eta \delta_m]$, so $\| \Ib - \eta \Hb  \| \le 1 + \eta \delta_m$.
This implies $\| \hat\bpsi^{k-1} \| \le \| \bpsi^{k-1} \|$. 

On the other hand,  for all $k\geq 1$, we have
\begin{eqnarray}
\EE \left[ \hat\bpsi^{k-1}\,^\top\bzeta^k_d\cdot \cI_{\|\bpsi^{k-1}\|\leq 2 } \mid \cF^{k-1}\right] \overset{a}{=} \cI_{\|\bpsi^{k-1}\|\leq 2 } \cdot\EE [\hat\bpsi^{k-1}\,^\top \bzeta^k_d \mid \cF^{k-1}] =0,
\end{eqnarray}
and
\begin{eqnarray}
\EE \left[ \left|\hat\bpsi^{k-1}\,^\top \bzeta^k_d \cdot \cI_{\|\bpsi^{k-1}\|\leq 2 }\right|^2 \mid \cF^{k-1}\right] \overset{a~\&~\eqref{hatepk}}{=} \cI_{\|\bpsi^{k-1}\|\leq 2 } \cdot 2L\|\bpsi^{k-1} \|^2 \leq 8L,
\end{eqnarray} 
where $\cI$ denotes the indicator function, $\overset{a}=$ uses $\bpsi^{k-1}$ and $\hat\bpsi^{k-1}$ are measurable on $\cF^{k-1}$.   By the  standard Azuma's inequality,  with probability $1 - 0.1/(2K_0)$, for any $l$ from $1$ to $K_0$, 
\begin{eqnarray}\label{probability-2}
\left|\sum_{k=1}^l   \hat\bpsi^{k-1}\,^\top\bzeta^k_d \cdot \cI_{\|\bpsi^{k-1}\|\leq 2}
\right|\leq 4\sqrt{L l  \log(40K_0)} \leq 4\sqrt{L K_0  \log(40K_0)}\overset{\eqref{parametersetting}}{\leq} \frac{1}{\eta}.
\end{eqnarray}

Analogously, we also have
\begin{eqnarray}
\EE \left[ \e_1^\top \bzeta_d^k  \cdot \cI_{\|\bpsi^{k-1}\|\leq 2 }  \mid \cF^{k-1}\right]=\mathbf{0}, \quad  \EE \left[ \left|\e_1^\top \bzeta_d^k  \cdot \cI_{\|\bpsi^{k-1}\|\leq 2 }\right|^2 \mid \cF^{k-1}\right]\leq 4L.
\end{eqnarray}
Thus with standard Azuma's inequality,
\begin{eqnarray}\label{probability-1}
\left|\sum_{k=1}^l \e_1^\top \bzeta_d^k  \cdot \cI_{\|\bpsi^{k-1}\|\leq 2 } 
\right|\leq \sqrt{ 8 L l  \log(40k_0)} \leq \sqrt{8L K_0  \log(12K_0/p)}\overset{\eqref{parametersetting}}{\leq} \frac{1}{4\eta}
\end{eqnarray}
happens with probability at least $1-0.1/(2K_0)$.

So by union bound, there exists a high-probability event $\cH_o$  happening with probability at least $0.9$ such that the following  inequalities hold for each $l=1,2,\dots, K_0$,

\begin{eqnarray}\label{probability-3}
\left|\sum_{k=1}^l   \hat\bpsi^{k-1}\,^\top\bzeta^k_d \cdot \cI_{\|\bpsi^{k-1}\|\leq 2}
\right|\leq  \frac{1}{\eta}, \quad  \left|\sum_{k=1}^l \e_1^\top \bzeta_d^k  \cdot \cI_{\|\bpsi^{k-1}\|\leq 2 } 
\right| \leq  \frac{1}{4\eta}.
\end{eqnarray}

On the other hand,  we have from \eqref{psitexpress} and \eqref{hatpsi} that for all $k\geq 1$,
\begin{align*}
\| \bpsi^k \|^2
 &= 
\left\|
\frac{\left[ \Ib - \eta \Hb  \right]}{1 + \eta \delta_m} \bpsi^{k-1}
+
\eta \hat\Db^{k-1} \bpsi^{k-1}
+ 
\eta \bzeta_d^{k} 
\right\|^2
 \\&=
\|\hat\bpsi^{k-1}\|^2
+
2\eta \hat\bpsi^{k-1}\,^\top
\hat\Db^{k-1} \bpsi^{k-1}
+
\eta^2 \left\|
\hat\Db^{k-1} \bpsi^{k-1}
+ 
\bzeta_d^{k} 
\right\|^2
+ 
2\eta \hat\bpsi^{k-1}\,^\top
\bzeta_d^{k} 
 \\&\le
\|\bpsi^{k-1}\|^2
+
Q_{1,k} + 
Q_{2,k} + 
Q_{3,k}
\end{align*}
Hence from \eqref{eqhatD},
\begin{align*}
Q_{1,k}
 &=
2\eta \hat\bpsi^{k-1}\,^\top
\hat\Db^{k-1} \bpsi^{k-1}
 \le
2\eta \cdot \rho \radius \| \bpsi^{k-1} \|^2
\end{align*}
and
\begin{align*}
Q_{2,k}
 &= 
\eta^2 \left\|
\hat\Db^{k-1} \bpsi^{k-1}
+ 
\bzeta_d^{k-1} 
\right\|^2
 \\&\le
2\eta^2  \| \hat\Db^{k-1} \bpsi^{k-1} \|^2 + 2\eta^2  \|\bzeta_d^{k-1} \|^2
 \\&\le
2\eta^2 \cdot \rho^2 \radius^2 \| \bpsi^{k-1} \|^2 + 8\eta^2 L^2 \| \bpsi^{k-1}  \|^2
 \\&\le
16\eta^2 \cdot L^2 \| \bpsi^{k-1} \|^2
,
\end{align*}
and
\begin{align*}
Q_{3,k}
 &=
2\eta \hat\bpsi^{k-1}\,^\top \bzeta_d^{k-1} 
.
\end{align*}
Under the event $\cH_0$ happens, by induction, when $k=0$, $\| \bpsi^{0}\| = \| \e_1\|\leq 2$, suppose $\| \bpsi^{l}\|\leq 2$ holds for all $l = 0$ to $k-1$,  we have  for the step $k$,
\begin{align*}
\| \bpsi^k \|^2
 &\le
\| \bpsi^0 \|^2
 +
\sum_{s=1}^k Q_{1,s}
 + 
\sum_{s=1}^k Q_{2,s}
 + 
\sum_{s=1}^k Q_{3,s}
 \\&\le
1
+
2\eta 
\sum_{s=1}^{k} 
\rho \radius \| \bpsi^{s-1} \|^2
+
16\eta^2 \cdot L^2
\sum_{s=1}^k  \| \bpsi^{s}  \|^2
+
2\eta 
\sum_{s=1}^k 
\hat\bpsi^{s-1}\,^\top \bzeta_d^s
 \\&{\le}
1
+
2 \rho \radius \cdot 4 \cdot \eta k
+
16\eta^2 \cdot L^2 \cdot 4 \cdot k
+
2\eta 
\sum_{s=1}^{k} 
\hat\bpsi^{s-1}\,^\top \bzeta_d^s\cdot \cI_{\|\bpsi^{s-1} \|\leq 2}
 \\&\overset{a}{\leq}
1
+
16 \rho \radius \cdot \eta k
+
2\eta 
\sum_{s=1}^{k} 
\hat\bpsi^{s-1}\,^\top \bzeta_d^s\cdot \cI_{\|\bpsi^{s-1} \|\leq 2}
 \le
1+1+2 = 4
,
\end{align*}
where   $\overset{a}\leq$ uses $\eta\leq\frac{\rho \radius}{8 L^2 }$ (because \eqref{parametersetting} and $B\leq \frac{1}{L}$).
This conclude the proof of \eqref{psinorm_express}. For $\e_1^\top \bpsi^k$, we have
\begin{align*}
\e_1^\top \bpsi^k
&=
\e_1^\top \bpsi^0
+
\sum_{s=0}^{k-1} 
\eta \e_1^\top \hat\Db_{s} \bpsi^{1}
+ 
\sum_{s=0}^{k-1} 
\eta  \e_1^\top \bzeta_d^s
\\&\ge
1
-
 \eta \sum_{s=0}^{k-1} 
2\rho  \radius \cdot  \|\bpsi^{s-1} \| 
+ 
\eta \sum_{s=1}^{k} 
 \e_1^\top \bzeta_d^s\cdot {\cI_{\|\bpsi^{s-1} \|\leq 2}}
\\&\ge
1
-
 \eta \cdot k \cdot 2\rho  \radius \cdot 2
+ 
\eta \sum_{s=0}^{k-1} 
 \e_1^\top \bzeta_d^s\cdot {\cI_{\|\bpsi^{s-1} \|\leq 2}}
\geq
1 - \frac18 - \frac28
 >
\frac12
,
\end{align*}
concluding \eqref{psi1_express}, and hence the lemma.
\end{proof}

\item
Now, we have all the ingredients necessary to prove our final lemma.

\begin{proof}[Proof of Lemma \ref{lemm:key}]
Recall that the deterministic time $\Ko$ was defined in \eqref{Todef}, we have on the event $(\cK_1 > \Ko)$ that $\z^{\Ko} = \bar{\w}^{\Ko} - \w^{\Ko}$ and hence $\|\z^{\Ko} \| \le \|\bar{\w}^{\Ko}\| + \|\w^{\Ko}\| \le 2\radius$, which concludes
\beq\label{setrl1}
(\cK_1 > \Ko) \subseteq (\|\z^{\Ko}\| \le 2\radius)
.
\eeq
In the mean time, from \eqref{yk} we know that on the event $(\cK_1 > \Ko)$, $\z^k = \bar{\w}^k - \w^k$ for all $k\le \Ko$, from \eqref{psi1_express} 
$$
\e_1^\top \bpsi^{\Ko}
 >
\frac12
.
$$
So on the event $(\cK_1 > \Ko) \cap \cH_{0}$
\begin{align*}
\| \z^{\Ko} \|
 &=
q \left(1 + \eta \delta_m \right)^{\Ko} \| \bpsi^{\Ko} \|
 \ge
q_0 \left(1 + \eta (\epH) \right)^{\Ko} | \e_1^\top \bpsi^{\Ko} |
 \\&>
q_0\cdot \frac{6\radius}{q_0} \cdot \frac12
=
3\radius
,
\end{align*}
so
\beq\label{setrl2}
(\cK_1 > \Ko) \cap \cH_{0} \subseteq (\| \z^{\Ko} \| > 3\radius)
\eeq
Combining \eqref{setrl1}, \eqref{setrl2} and the fact that $\radius > 0$ gives
$$
(\cK_1 > \Ko) \cap \cH_{0} 
\subseteq (\|\z^{\Ko}\| \le 2\radius) \cap \cH_{0}  \cap (\| \z^{\Ko} \| > 3\radius)
= \varnothing
,
$$
and hence
$
(\cK_1 > \Ko) \subseteq \cH_{0}^c
$
which leads to
$$
\PP( \cK_1 > \Ko ) \le \PP( \cH_{0}^c ) \le 0.1
,
$$
proving \eqref{cT1bdd}.
Hence \eqref{kep_rep} and Lemma \ref{lemm:key} hold.

\end{proof}

\end{enumerate}

\section{Deferred Proofs of Part \mbox{II}: Faster Descent}\label{prooftwo}

\subsection{Definition and Preliminary}
  In Part \mbox{II},  we still use $\bH$ to denote $\nabla^2 f(\x^{0})$ and let $$\bxi^{k+1} = \nabla F(\x^k, \bzeta^{k+1}) - \nabla f(\x^k), \quad k\geq 0.$$
  
  Recall the definition of  $\cS$, $\cP_{\cS}$, $\cS\bot$, and $\cP_{\cS\bot}$ in   Appendix \ref{2Sufficient Descent}. 
  Let $\uu^k = \cP_{\cS} \left(\x^k - \x^{0}\right)$, and $\vv^k = \cP_{\cS\bot}\left(\x^k - \x^{0}\right)$.
 We can decompose the update equation of SGD as: 
\begin{eqnarray}
\uu^{k+1} &=& \uu^{k} - \eta \cP_{\cS}\nabla f\left(\x^k\right) - \eta  \cP_{\cS}\bxi^{k+1}. \label{update u}\\
\vv^{k+1}  &=& \vv^{k} - \eta \cP_{\cS\bot}\nabla f\left(\x^k\right) - \eta  \cP_{\cS\bot}\bxi^{k+1},  \label{update v}
\end{eqnarray}
with $k \geq 0$. And $\uu^{0} = \mathbf{0}$, $\vv^{0} = \mathbf{0}$.
From the definition of  $g(\x)$ in Appendix  \ref{2Sufficient Descent},  we have
\begin{eqnarray}
g(\x) &=&  \left[\nabla f(\x^{0})\right]^\bT\left(\x - \x^{0}\right) +  \frac{1}{2}\left[\x - \x^{0}\right]^\bT \bH \left[\x - \x^{0}\right] \\
&=&  \left[\nabla f(\x^{0})\right]^\bT\left[\cP_{\cS} \left(\x - \x^{0}\right) + \cP_{\cS\bot} \left(\x - \x^{0}\right)\right]+ \frac{1}{2}\left[\cP_{\cS} \left(\x - \x^{0}\right)\right]^\bT \bH \left[\cP_{\cS} \left(\x - \x^{0}\right)\right]\notag\\
&& +  \frac{1}{2}\left[\cP_{\cS\bot} \left(\x - \x^{0}\right)\right]^\bT \bH \left[\cP_{\cS\bot} \left(\x - \x^{0}\right)\right]\notag\\
&=& \left[\cP_{\cS} \nabla f(\x^{0})\right]^\bT\left[\cP_{\cS} \left(\x - \x^{0}\right)\right]+ \left[\cP_{\cS\bot} \nabla f(\x^{0})\right]^\bT\left[\cP_{\cS\bot} \left(\x - \x^{0}\right)\right]\notag\\
&& + \frac{1}{2}\left[\cP_{\cS} \left(\x - \x^{0}\right)\right]^\bT \bH_{\cS} \left[\cP_{\cS} \left(\x - \x^{0}\right)\right] +  \frac{1}{2}\left[\cP_{\cS\bot} \left(\x - \x^{0}\right)\right]^\bT \bH_{\cS\bot} \left[\cP_{\cS\bot} \left(\x - \x^{0}\right)\right], \notag
\end{eqnarray} 
where in the last equality  we use $\cP^2_{\cS} = \cP_{\cS}$ and $\cP^2_{\cS\bot} = \cP_{\cS\bot}$, because $\cP_{\cS}$ and $\cP_{\cS\bot}$ are projection matrices.
Thus if  $\uu = \cP_{\cS} (\x-\x^0)$ and $\vv = \cP_{\cS\bot} (\x-\x^0)$,  we have
\begin{eqnarray}
g(\x) = g_{\cS}(\uu) + g_{\cS\bot}(\vv).
\end{eqnarray} 
For clarify,  we denote $\nabla_{\uu} f(\x^k) = \cP_{\cS}\nabla f(\x^k)$, and  $\nabla_{\vv} f(\x^k)  = \cP_{\cS\bot}\nabla f(\x^k)$, respectively. Similarly, let $\bxi_{\uu}^k = \cP_{\cS}\bxi^k$, and  $\bxi_{\vv}^k = \cP_{\cS\bot}\bxi^k$.   In the following, we denote $\sK = \sK_0\wedge K_0$ which is also a stopping time.  The  Lemma below  is basic to obtain our result.

\begin{lemma}
Given $\x^0$, for any $\x$, if $\left\| \x - \x^{0}\right\| \leq B$, then
	\begin{eqnarray}\label{basic1}
	\|\nabla f(\x) - \nabla g(\x) \| \leq \rho B^2/2.
	\end{eqnarray}
  For any symmetric matrix $\A$, with $0< a \leq \frac{1}{\|\A\|_2}$,  for any $i=0,1,\dots$, and $j= 0,1,\dots$, we have
  \begin{eqnarray}\label{basic2}
  \left\|(\I - a \A)^i \A (\I - a \A)^j\right\|_2 \leq \frac{1}{a(i+j+1)}.
  \end{eqnarray}
\end{lemma}
\begin{proof}
	For \eqref{basic1}, we have
	\begin{eqnarray}
	&&\left\|\nabla f(\x) - \nabla g(\x) \right\|\notag\\
	&=&\left\|\nabla f(\x) - \nabla f(\x^{0}) - \bH (\x - \x^{0}) \right\|\notag\\
	&=&\left\|\left(\int_{0}^1 \nabla^2 f(\x^{0} + \theta (\x - \x^0)) d_{\theta} - \bH\right)  \left(\x - \x^{0}\right) \right\|\notag\\
	&\overset{a}\leq& \left\|\left(\int_{0}^1  \rho \theta \left\| \x - \x^0\right\| d_{\theta} \right) \right\|   \left\|\x - \x^{0}\right \|\notag\\
	&\leq& \rho B^2/2
	\end{eqnarray} 
	where in $\overset{a}\leq$, we use \eqref{HessSmooth} that $f(\x)$ has $\rho$-Lipschitz continuous Hessian.
	
	\eqref{basic2} is from \cite{jin2017escape}. To prove it, suppose the eigenvalue of $\{\A\}$  is $\{\lambda_l\}$, thus the  eigenvalue of $(\I - a \A)^i\bH (\I - a \A)^{j}$ is $\{\lambda_l (1 - a \lambda_l)^{i+j} \}$. For the function of $\lambda (1 - a \lambda)^{i+j}$, we can compute out its derivative as $ (1 - a\lambda)^{i+j} - (i+j)a \lambda (1 - a \lambda)^{i+j-1}$. Then with simple analysis, we can find that  the maximal point is obtained only at $\frac{1}{(1+i+j)a}$. If $i=0$ and $j=0$, \eqref{basic2} clearly holds. Otherwise, we have
\begin{eqnarray}
\left\|(\I - a \A)^i \A (\I - a \A)^j\right\|_2 \leq \frac{1}{(1+j+i)a} \left(1 - \frac{1}{1+j+i}\right)^{j+i} \leq \frac{1}{(1+j+i)a}.
\end{eqnarray}
\end{proof}

\subsection{Analysis on Quadratic  Approximation}\label{quadratic}
As have been introduced  before, our main technique to obtain a faster $\tO(\ep^{-3.5})$ convergence rate is by separately analyzing  the two quadratic  approximations: $g_{\cS}(\cdot)$ and $g_{\cS\bot}(\cdot)$. We will show in this section that the noise effect can be upper bounded by $\tO(\ep^{1.5})$ instead of $\cO(\ep)$ via our tool.

We first   summarize our result for $g_{\cS}(\cdot)$  in the following lemma:
\begin{lemma}\label{analysis on uu}
	Set hyper-parameters  in \eqref{parametersetting}  for  Algorithm \ref{algo:SGD}. With probability at least $1 - p/4$,    we have
	\begin{eqnarray}\label{uu end4}
	\!\!\!\!\!\!\!\!g_{\cS}(\uu^{\sK})  - g_{\cS} (\uu^{0})&\leq&-\frac{25\eta\sum_{k=0}^{\sK-1}\left\|  \nabla  g_{\cS} \left(\y^k\right)\right\|^2  }{32} +4\eta \sigma^2 \left(\log(K_0)+3\right) \log(48K_0/p) + \eta \rho^2 B^4 K_0\notag\\
	&=&-\frac{25\eta\sum_{k=0}^{\sK-1}\left\|  \nabla  g_{\cS} \left(\y^k\right)\right\|^2  }{32} + \tO\left(\ep^{1.5}\right).
	\end{eqnarray}
\end{lemma}

\noindent\textbf{Proofs of Lemma \ref{analysis on uu}}\\
Our novel technique to analyze $g_{\cS}(\cdot)$ is by first considering an auxiliary Gradient Descent trajectory, which performs update as:
\begin{eqnarray}\label{y update}
\y^{k+1} = \y^k - \eta \nabla g_{\cS}\left(\y^k\right), \quad k\geq 0.
\end{eqnarray}
and $\y^{0} = \uu^{0}$.    $\y^{k}$ preforms Gradient Decent on $ g_{\cS} (\cdot)$,  which is deterministic given $\x^0$. We  study the property of $\y^{\sK}$ and obtain the following standard results:
\begin{itemize}
	\item Because $g_{\cS}(\cdot)$ has $L$-Lipschitz continuous gradient ($\| \bH_{\cS}\|_2 \leq L$), we have
	\begin{eqnarray}\label{update y}
	&&g_{\cS}\left(\y^{k+1}\right)\notag\\
	&\leq&g_{\cS}\left(\y^{k}\right) + \< \nabla g_{\cS}\left(\y^{k}\right), \y^{k+1} - \y^k  \> +  \frac{L}{2}\left\|\y^{k+1} - \y^{k}\right\|^2\notag\\
	&\overset{\eqref{y update}}{=}& g_{\cS} \left(\y^k\right)  - \eta \left(1 - \frac{L\eta}{2}\right)\left\| \nabla g_{\cS}\left(\y^{k}\right)\right\|^2.
	\end{eqnarray}
	\item By telescoping \eqref{update y} from $0$ to $\sK-1$, we have
	\begin{eqnarray}\label{uu end2}
	g_{\cS}\left(\y^{\sK}\right)&\leq& g_{\cS}\left(\y^{0}\right) - \eta\left(1 - \frac{L\eta}{2}\right) \sum_{k = 0}^{\sK-1}\left\| \nabla g_{\cS}\left(\y^{k}\right)\right\|^2\notag\\
	& \overset{L \eta \leq \frac{1}{16}}{\leq}& g_{\cS}\left(\y^{0}\right) - \frac{31\eta}{32} \sum_{k = 0}^{\sK-1}\left\| \nabla g_{\cS}\left(\y^{k}\right)\right\|^2.
	\end{eqnarray}
\end{itemize}

To obtain Lemma \ref{analysis on uu}, 
 we  bound the difference between $\uu^\sK $ and $\y^{\sK}$. Define
$$\z^{k} \coloneqq \uu^k - \y^k.$$
The remaining is to conclude the  properties of $\z^{\sK}$, stated as follows:
\begin{lemma}\label{zz}
	With probability at least  $1 - p/6$, we have
	\begin{eqnarray}\label{z4}
	\left\| \z^{\sK}\right\|\leq \frac{3B}{32}\asymp \ep^{0.5}.
	\end{eqnarray}
	and 
	\begin{eqnarray}\label{z5}
	\left(\z^{\sK}\right)^\bT \bH_{\cS} \left(\z^{\sK}\right) \leq   8\sigma^2 \eta \left(\log(K_0)+1\right) \log(48K_0/p) + \eta \rho^2 B^4 K_0\asymp \ep^{1.5}.
	\end{eqnarray}
\end{lemma}

\begin{proof}[Proofs of Lemma \ref{zz}]
With $\z^{k} = \uu^k - \y^k$ being the difference iteration, we have
	\begin{eqnarray}\label{update of z}
	\z^{k+1} &=& \z^{k} - \eta \left(\nabla g_{\cS}\left(\uu^k\right)-\nabla g_{\cS}\left(\y^k\right)\right) -  \eta\left(\nabla_{\uu}f\left(\x^k\right) - \nabla g_{\cS}\left(\uu^k\right)\right) - \eta \bxi^{k+1}_{\uu}\notag\\
	&=& (\I - \eta \bH_{\cS} )\z^k - \eta\left(\nabla_{\uu}f\left(\x^k\right) - \nabla g_{\cS}\left(\uu^k\right)\right)- \eta \bxi^{k+1}_{\uu}, \quad k\geq 0.
	\end{eqnarray}
	And $\z^{0} = \mathbf{0}$.  Thus we can obtain the general solution of \eqref{update of z} as
	\begin{eqnarray}
	\z^k =  - \sum_{j =1}^{k}\eta (\I - \eta \bH_{\cS})^{k - j} \bxi_{\uu}^j  - \eta \sum_{j =0}^{k-1} (\I - \eta \bH_{\cS})^{k-1 - j} \left(\nabla_{\uu}f\left(\x^j\right) - \nabla g_{\cS}\left(\uu^j\right)\right), \quad k\geq 0.
	\end{eqnarray}
	
Setting $k = \sK$, by triangle inequality,  we have 
	\begin{eqnarray}\label{z3}
	\left\|\z^{\sK} \right\| \leq   \left\| \sum_{j =1}^{\sK}\eta (\I - \eta \bH_{\cS})^{\sK - j} \bxi_{\uu}^j\right\| + \left\| \eta \sum_{j =0}^{\sK-1} (\I - \eta \bH_{\cS})^{\sK-1 - j} \left(\nabla_{\uu}f\left(\x^j\right) - \nabla g_{\cS}\left(\uu^j\right)\right) \right\|.
	\end{eqnarray}
We separately bound the two terms in the right hand sides of  	\eqref{z3}. For the first term, for any fixed $l$ from $1$ to $K_0$, and any j from  $1$ to $l$,   we have 
	\begin{eqnarray}\label{z1}
 \EE \left[\eta (\I - \eta \bH_{\cS})^{l - j} \bxi_{\uu}^j\mid \cF^{j-1}\right]\overset{a}{=} \mathbf{0},\quad	\left\| \eta (\I - \eta \bH_{\cS})^{l - j} \bxi_{\uu}^j\right\| \overset{b}{\leq} \eta \sigma,
	\end{eqnarray}
	where $\overset{a}=$ uses that $\| \bxi^j_{\uu}\| = \| \cP_\cS\bxi^j\|$,   $\overset{b}\leq$ further uses $ \| \bxi^j_{\uu}\|\leq \| \bxi^j\|\leq \sigma$,  because $\cP $ is projection matrix, and the  the bounded noise  assumption in \eqref{bounded} and $\|(\I -  \eta \bH_{\cS})^{l - j} \|_2 \leq 1$ for all $j$ from $1$ to $l$.
	 Thus by the Vector-Martingale Concentration Inequality in Theorem \ref{proh}, we have with probability $1 - p/(12K_0)$, 
	\begin{eqnarray}\label{probability1}
	\left\| \sum_{j =1}^{l}\eta (\I - \eta \bH_{\cS})^{l - j} \bxi_{\uu}^j\right\|\leq 2\eta \sigma\sqrt{ (l) \log(48K_0/p)} \leq 2\eta \sigma\sqrt{ K_0 \log(48K_0/p)}\overset{\eqref{parametersetting}}{\leq} \frac{B}{16}.
	\end{eqnarray}
By union bound, with  probability  at least $1-p/12$, \eqref{probability1} holds for all $l$ from $1$ to $K_0$. Because  $1\leq \sK \leq K_0$, with probability  at least $1-p/12$,
	\begin{eqnarray}\label{probability1tt}
\left\| \sum_{j =1}^{\sK}\eta (\I - \eta \bH_{\cS})^{\sK - j} \bxi_{\uu}^j\right\|\leq  \frac{B}{16}.
\end{eqnarray}

	For the  second term in the right hand side of \eqref{z3},  we  have 
	\begin{eqnarray}\label{z2}
	 \left\| \eta \sum_{j =0}^{\sK-1} (\I - \eta \bH_{\cS})^{\sK-1 - j} \left(\nabla_{\uu}f\left(\x^j\right) - \nabla g_{\cS}\left(\uu^j\right)\right) \right\|&\overset{a}{\leq}& \eta \sum_{j =0}^{\sK-1}  \left\| \left(\nabla_{\uu}f\left(\x^j\right) - \nabla g_{\cS}\left(\uu^j\right)\right) \right\|\notag\\
	&\overset{b}{\leq}&\eta \sum_{j =0}^{\sK-1} \left\|\nabla f\left(\x^j\right) - \nabla g\left(\x^j\right)  \right\|\notag\\
	&\overset{\eqref{basic1}}{\leq } & \frac{\rho\eta B^2 K_0}{2} \overset{\eqref{parametersetting}}{\leq} \frac{B}{32},
	\end{eqnarray}
	where in $\overset{a}\leq$, we use triangle inequality, and $\left\|\mathbf{I} -\eta \bH_{\cS} \right\|^{\sK-1-j}_2 \leq 1$ with $j$ from $0$ to $\sK-1$; $\overset{b}\leq$ uses $\|\cP_\cS (\nabla f(\x)  - \nabla g(\x)) \| \leq \|\nabla f(\x)  - \nabla g(\x) \|$ becuase $\cP_\cS$ is projected matrix.  Substituting \eqref{probability1tt} and \eqref{z2} into \eqref{z3}, we obtain \eqref{z4}.

	To prove \eqref{z5}, using the fact that  $(\a + \b)^\bT \A (\a+ \b) \leq 2 \a^{\bT}\A \a +2 \b^{\bT}\A \b  $ holds for  any  symmetry positive  definite matrix $\A$, we have
	\begin{eqnarray}\label{zend1}
	&&\left(\z^{\sK}\right)^\bT \bH_{\cS} \left(\z^{\sK}\right) \\
	&\leq& 2\eta^2 \left(\sum_{j =1}^{\sK} (\I - \eta \bH_{\cS})^{\sK - j} \bxi_{\uu}^j \right)^\bT \bH_{\cS} \left(\sum_{j =1}^{\sK} (\I - \eta \bH_{\cS})^{\sK - j} \bxi_{\uu}^j \right)\notag\\
	&&\!\!\!\!\!\!\!\!\!\!\!\!\!\!\!\!+ 2\eta^2 \left( \sum_{j =0}^{\sK-1} (\I - \eta \bH_{\cS})^{\sK-1 - j} \left(\nabla_{\uu}f\left(\x^j\right) - \nabla g_{\cS}\left(\uu^j\right)\right)\right)^{\bT}\bH_{\cS} \left( \sum_{j =0}^{\sK-1} (\I - \eta \bH_{\cS})^{\sK-1 - j} \left(\nabla_{\uu}f\left(\x^j\right) - \nabla g_{\cS}\left(\uu^j\right)\right)\right)\notag\\
	&=& 2\left\| \eta\left(\sum_{j =1}^{\sK} \bH^{1/2}_{\cS}(\I - \eta \bH_{\cS})^{\sK - j} \bxi_{\uu}^j \right)   \right\|^2\notag\\
	&& + 2\eta^2 \sum_{j =0}^{\sK-1}\sum_{l =0}^{\sK-1}\left(\nabla_{\uu}f\left(\x^j\right) - \nabla g_{\cS}\left(\uu^j\right)\right)^\bT (\I - \eta \bH_{\cS})^{\sK -1- j}\bH_{\cS}(\I - \eta \bH_{\cS})^{\sK-1 - l}\left(\nabla_{\uu}f\left(\x^j\right) - \nabla g_{\cS}\left(\uu^j\right)\right)\notag\\
	&\overset{\eqref{basic1}}{\leq}&2\left\| \eta\left(\sum_{j =1}^{\sK}  \bH^{1/2}_{\cS}(\I - \eta \bH_{\cS})^{\sK - j} \bxi_{\uu}^j \right)   \right\|^2 + 2\eta^2\frac{\rho^2 B^4}{4}\sum_{j =0}^{\sK-1}\sum_{l=0}^{\sK-1} \left\|  (\I - \eta \bH_{\cS})^{\sK-1 - j}\bH_{\cS}(\I - \eta \bH_{\cS})^{\sK-1 - l} \right\|_2.\notag
	\end{eqnarray}
	
For 	the first term in the right hand side  of $\eqref{zend1}$,  for any fixed $l$ from $1$ to $K_0$, and any j from  $1$ to $l$,   we have 
$$  \EE\left[\eta\left( \bH^{1/2}_{\cS}(\I - \eta \bH_{\cS})^{l - j} \bxi_{\uu}^j \right)\mid \cF^{j-1} \right] = \mathbf{0}, $$
and
	\begin{eqnarray}
	&&\left\|\eta\left( \bH^{1/2}_{\cS}(\I - \eta \bH_{\cS})^{l - j} \bxi_{\uu}^j \right)\right\|^2\notag\\
	&  \leq&  \eta^2\|\bxi_{\uu}^j\|  \left\| (\I - \eta  \bH_{\cS})^{l - j}  \bH_{\cS}  (\I - \eta  \bH_{\cS})^{l - j}\right \|_2\|\bxi_{\uu}^j  \| \overset{\eqref{basic2}~\&~\| \bxi^j_\uu\|\leq \sigma} {\leq}  \frac{\eta\sigma^2}{1+2 (l-j)},
	\end{eqnarray}
 by  the Vector-Martingale Concentration Inequality in \S \ref{proh},   we have with probability $1 - p/(12K_0)$
	\begin{eqnarray}\label{probability2}
	&&\left\| \eta\left(\sum_{j =1}^{l} \cH^{1/2}(\I - \eta \bH_{\cS})^{l- j} \bxi_{\uu}^j \right)   \right\|^2 \leq   4 \eta \sigma^2\log(48K_0/p) \sum_{j= 1}^{l} \frac{1}{1+2(l-j)}   \notag\\
	&\leq& 4\sigma^2 \eta   \log(48K_0/p)\sum_{j= 0}^{K_0-1} \frac{1}{1+j}   \leq 4\sigma^2 \eta \left(\log(K_0)+1\right) \log(48K_0/p).
	\end{eqnarray}
By union bound, with  probability  at least $1-p/12$, \eqref{probability2} holds for all $l$ from $1$ to $K_0$. Because  $1\leq\sK\leq K_0$, with probability  at least $1-p/12$,
	\begin{eqnarray}\label{probability2tt}
&&\left\| \eta\left(\sum_{j =1}^{\sK} \cH^{1/2}(\I - \eta \bH_{\cS})^{\sK- j} \bxi_{\uu}^j \right)   \right\|^2 \leq 4\sigma^2 \eta \left(\log(K_0)+1\right) \log(48K_0/p).
\end{eqnarray}

	For the second term in the right hand side of $\eqref{zend1}$,  we have
	\begin{eqnarray}\label{z7}
	&&\eta^2\frac{\rho^2 B^4}{4}\sum_{j =0}^{\sK-1}\sum_{l =0}^{\sK-1} \left\|  (\I - \eta \bH_{\cS})^{\sK-1- j}\bH(\I - \eta \bH_{\cS})^{\sK-1 - l} \right\|_2 \notag\\
	&\overset{\eqref{basic1}}{\leq}&\eta\frac{\rho^2 B^4}{4} \sum_{j =0}^{\sK-1}\sum_{l =0}^{\sK-1}  \frac{1}{1+ (\sK-1 - j)+(\sK-1 - l) }\overset{\sK \leq K_0}{\leq}\eta\frac{\rho^2 B^4}{4} \sum_{j =0}^{K_0-1}\sum_{l =0}^{K_0-1}  \frac{1}{1+ j +l} \notag\\
	&=&\eta\frac{\rho^2 B^4}{4} \sum_{j=0}^{2(K_0-1)} \frac{\min(1+j, 2K_0-1 - j)}{1+j}\leq \frac{\eta\rho^2 B^4 K_0}{2}.
	\end{eqnarray}
	Substituting \eqref{probability2tt} and \eqref{z7} into \eqref{zend1}, we obtain \eqref{z5}.
\end{proof}

\begin{proof}[Proofs of Lemma \ref{analysis on uu}]
 Let $\y^* = \argmin_{\y} g_{\cS} (\y)$. By the optimal condition of $\y^*$, we have
$$\nabla_{\cS} f\left(\x^{0}\right) =   -\bH_{\cS} \y^*.$$

Define $\ty^k =  \y^{k} - \y^{*}$. From the update rule of $\y^k$ in \eqref{y update}, we have
\begin{eqnarray}
\bH_{\cS}\ty^k  &=& \nabla g_{\cS}\left( \y^k\right),\label{g ty}\\
\ty^{k+1} &=& \ty^k  - \eta \bH_{\cS} \ty^k. \label{update ty}
\end{eqnarray}
 We can bound the first-order difference between $g_{\cS} \left(\uu^{\sK}\right)$ and $g_{\cS} \left(\y^{\sK}\right)$ as:
\begin{eqnarray}\label{z8}
&&\< \nabla g_{\cS} \left(\y^{\sK}\right),  \uu^{\sK} - \y^{\sK} \>\overset{\eqref{g ty}}{=}\<\ty^{\sK}, \z^{\sK} \>_{\bH_{\cS}}\\
&\overset{\eqref{update ty} ~ \eqref{update of z}}{=}& \< (\I -  \eta \bH_{\cS}) \ty^{\sK-1}  , (\I -  \eta \bH_{\cS})\z^{\sK-1}  - \eta \bxi_{\uu}^{\sK} -  \eta  \left(\nabla_{\uu} f\left(\x^{\sK-1}\right) - \nabla g_{\cS}\left(\uu^{\sK-1}\right)\right)  \>_{\bH_{\cS}}\notag\\
&\overset{a}=& \<\ty^{\sK-1}, \z^{\sK-1} \>_{\bH_{\cS} (\I -  \eta \bH_{\cS})^2}  - \eta  \<\ty^{\sK-1},  \bxi_{\uu}^{\sK}  \>_{\bH_{\cS} (\I -  \eta \bH_{\cS})}\notag\\
&& - \eta\< \ty^{\sK-1 }, \nabla_{\uu} f\left(\x^{\sK-1}\right) - \nabla g_{\cS}\left(\uu^{\sK-1}\right) \>_{\bH_{\cS} (\I -  \eta\bH_{\cS})}\notag\\
&\overset{b}=&\!\!\!\!\!\!\!\!\!\!\!\!\!\!\! - \eta  \sum_{k=1}^{\sK }\<\ty^{k-1},  \bxi_{\uu}^{k}  \>_{\bH_{\cS} (\I -  \eta \bH_{\cS})^{\sK - k+1}} - \eta \sum_{k=0}^{\sK-1 }\< \ty^{k}, \nabla_{\uu} f\left(\x^{k}\right) - \nabla g_{\cS}\left(\uu^{k}\right) \>_{\bH_{\cS} (\I -  \eta\bH_{\cS})^{\sK - k}},\notag
\end{eqnarray}
where in $\overset{a} = $, we use $(\I - \bH_{\cS})\bH_{\cS} = \bH_{\cS}(\I - \bH_{\cS})$, and in $\overset{b}=$, we use $\z^{0} = \mathbf{0}$.

We also bound the two terms in the right hand side of \eqref{z8}. For any fixed $l$ from $1$ to $K_0$, and any j from  $1$ to $l$,   we have
\begin{eqnarray}\label{z9}
&&\left|\<\ty^{j-1},  \bxi_{\uu}^{l}  \>_{\bH_{\cS} (\I -  \eta \bH_{\cS})^{l - j+1}}\right|^2=\left|\< \bH_{\cS} \ty^{j-1},  \bxi_{\uu}^{j}  \>_{ (\I -  \eta \bH_{\cS})^{l - j+1}}\right|^2 \notag\\
&\overset{\eqref{g ty}}{=}& \left|\< \nabla  g_{\cS}\left(\y^{j-1}\right),   \bxi_{\uu}^{j} \>_{ (\I -  \eta \bH_{\cS})^{l- j+1}}\right|^2\notag\\
&\overset{a}{\leq}& \sigma^2 \left\|  \nabla  g_{\cS}\left(\y^{j-1}\right)\right\|^2 \left\| (\I -  \eta \bH_{\cS})^{l-j+1}\right\|_2^2 \overset{b}{\leq} \sigma^2 \left\|  \nabla  g_{\cS}\left(\y^{j-1}\right)\right\|^2,
\end{eqnarray}
where $\overset{a}\leq$ uses $\| \bxi^j_{\uu}\| \leq \| \bxi^j\|\leq \sigma $, $\overset{b}\leq$ uses $\|(\I -  \eta \bH_{\cS})^{l - j+1} \|_2 \leq 1$ for all $j$ from $1$ to $l$. So for any $l$ from  $1\leq k \leq K_0$,  by standard Azuma–Hoeffding inequality, using $ \nabla  g_{\cS}\left(\y^k\right)$ is measurable on $\cF^{0}$,  with probability  at least $1-p/(12K_0)$,  we have
\begin{eqnarray}\label{protemp}
\left| \eta  \sum_{k=1}^{l}\<\ty^{k-1},  \bxi_{\uu}^{k}  \>_{\bH_{\cS} (\I -  \eta \bH_{\cS})^{l - k+1}}\right| &\leq&  \sqrt{2\eta^2 \sigma^2\log(24K_0/p) \sum_{k=0}^{l-1}\left\|  \nabla  g_{\cS} \left(\y^k\right)\right\|^2 } .
\end{eqnarray}
By union bound, with  probability  at least $1-p/12$, \eqref{protemp} holds for all $l$ from $1$ to $K_0$.  we have with probability at least $1-p/12$
\begin{eqnarray}\label{probability3}
\left| \eta  \sum_{k=1}^{\sK}\<\ty^{k-1},  \bxi_{\uu}^{k}  \>_{\bH_{\cS} (\I -  \eta \bH_{\cS})^{\sK - k+1}}\right| &\leq&  \sqrt{2\eta^2 \sigma^2\log(24K_0/p) \sum_{k=0}^{\sK-1}\left\|  \nabla  g_{\cS} (\y^k)\right\|^2 } \\
&\overset{a}\leq& \frac{\eta\sum_{k=0}^{\sK-1}\left\|  \nabla  g_{\cS} (\y^k)\right\|^2  }{16} +8\eta \sigma^2 \log(48K_0/p),\notag
\end{eqnarray}
where  in $\overset{a}\leq$, we use $\sqrt{ab}\leq \frac{a+b}{2}$ with $a\geq0$ and $b\geq0$.

For the second term in the right hand side  of  \eqref{z8}, we have
\begin{eqnarray}\label{z10}
&&\eta \sum_{k=0}^{\sK-1 }\< \ty^{k}, \nabla_{\uu} f\left(\x^{k}\right) - \nabla g_{\cS}\left(\uu^{k}\right) \>_{\bH_{\cS} (\I -  \eta \bH_{\cS})^{\sK - k}}\notag\\
&\overset{\eqref{g ty}}{=}&\eta \sum_{k=0}^{\sK-1 }\<   \nabla  g_{\cS} \left(\y^k\right), \nabla_{\uu} f\left(\x^{k}\right) - \nabla g_{\cS}\left(\uu^{k}\right) \>_{ (\I -  \eta \bH_{\cS})^{\sK - k}}\notag\\
&\overset{a}{\leq}&\eta \sum_{k=0}^{\sK-1 }\left\|  \nabla  g_{\cS} \left(\y^k\right)\right\|\cdot\left\| \nabla_{\uu} f\left(\x^{k}\right) - \nabla g_{\cS}\left(\uu^{k}\right) \right\| \notag\\
&\overset{b}{\leq}&  \frac{\eta\sum_{k=0}^{\sK-1}\left\|  \nabla  g_{\cS} \left(\y^k\right)\right\|^2  }{8} + 2\eta  \sum_{k=0}^{\sK-1}\left\|\nabla_{\uu} f\left(\x^{k}\right) - \nabla g_{\cS}\left(\uu^{k}\right) \right\|^2\notag\\
&\overset{\eqref{basic1}}{\leq}& \frac{\eta\sum_{k=0}^{\sK-1}\left\|  \nabla  g_{\cS} \left(\y^k\right)\right\|^2  }{8} + \eta \rho^2 B^4 K_0/2,
\end{eqnarray}
where $\overset{a}\leq$ uses $\left\|(\I -  \eta \bH_{\cS})^{\sK - k}\right\|_2\leq1$ with $k < \sK$, $\overset{b}\leq $ uses $ab\leq \frac{a^2+b^2}{2}$.

Substituting \eqref{probability3} and \eqref{z10} into \eqref{z8}, and using \eqref{z5}, we have
\begin{eqnarray}\label{uu end}
&&g_{\cS}\left(\uu^{\sK}\right) \\
&=&  g_{\cS}\left(\y^{\sK}\right)+ \< \nabla  g_{\cS}\left(\y^{\sK}\right), \uu^{\sK} - \y^{\sK} \> + \frac{1}{2}\left(\uu^{\sK} - \y^{\sK}\right)^\bT \bH \left(\uu^{\sK} - \y^{\sK}\right)\notag\\
&\overset{\eqref{z8}~ \eqref{z5}}\leq&g_{\cS}\left(\y^{\sK}\right)+ \frac{3\eta\sum_{k=0}^{\sK-1}\left\|  \nabla  g_{\cS} (\y^k)\right\|^2  }{16} +4\eta \sigma^2 (\log(K_0) +3) \log(48K_0/p) + \rho^2\eta B^4 K_0.\notag
\end{eqnarray}

Then by adding \eqref{uu end} and \eqref{uu end2}, we have
\begin{eqnarray}\label{uu end3}
&&g_{\cS}\left(\uu^{\sK}\right)  - g_{\cS} \left(\uu^{0}\right)\notag\\
&\leq&-\frac{25\eta\sum_{k=0}^{\sK-1}\left\|  \nabla  g_{\cS} \left(\y^k\right)\right\|^2  }{32} +4\eta \sigma^2 ( 3+\log(K_0)) \log(48K_0/p) + \rho^2\eta B^4 K_0,
\end{eqnarray}
implying Lemma \ref{analysis on uu}.
\end{proof}

We then investigate $g_{\cS\bot}(\cdot)$ and summarize its property as follows:
\begin{lemma}\label{analysis on vv}
	With hyper-parameters set in \eqref{parametersetting}  for  Algorithm \ref{algo:SGD},   we have
	\begin{eqnarray}\label{vv end5}
	g_{\cS\bot}(\vv^{\sK}) &\leq& g_{\cS\bot}(\vv^{0})  -  \sum_{k = 1}^{\sK}  \eta\< \nabla g_{\cS\bot}\left(\vv^{k-1}\right), \bxi_{\vv}^k\>  - \frac{7\eta}{8}\sum_{k = 0}^{\sK-1}  \left\|  \nabla g_{\cS\bot}\left(\x^{k}\right)  \right\|^2 + \frac{\rho^2 B^4 \eta K_0}{2}\notag\\
	&=&g_{\cS\bot}(\vv^{0})  -  \sum_{k = 1}^{\sK}  \eta\< \nabla g_{\cS\bot}\left(\vv^{k-1}\right), \bxi_{\vv}^k\> - \frac{7\eta}{8}\sum_{k = 0}^{\sK-1}  \left\|  \nabla g_{\cS\bot}\left(\x^{k}\right)  \right\|^2  +\tO(\ep^{1.5}).
	\end{eqnarray}
\end{lemma} 

\begin{proof}[Proofs of Lemma \ref{analysis on vv}] 
Lemma \ref{analysis on vv} can be obtained via the standard analysis. Specifically, from the definition of $g_{\cS\bot}(\cdot)$, we have
 \begin{eqnarray}\label{gcs}
 g_{\cS\bot}\left(\vv^{k+1}\right) &=&  g_{\cS\bot}\left(\vv^{k}\right) + \< \nabla g_{\cS\bot}\left(\vv^{k}\right), \vv^{k+1} - \vv^k  \> +  \left[\vv^{k+1} - \vv^{k}\right]^{\bT}\frac{\bH_{\cS\bot}}{2}\left[\vv^{k+1}-\vv^{k}\right]\notag\\
 &\overset{\bH_{\cS\bot}\preceq \mathbf{0}}{\leq}& g_{\cS\bot}\left(\vv^{k}\right) + \< \nabla g_{\cS\bot}\left(\vv^{k}\right) , \vv^{k+1} - \vv^k  \>\notag\\
 &\overset{\eqref{update v}}{=}& g_{\cS\bot}\left(\vv^k\right) -\eta \< \nabla g_{\cS\bot}\left(\vv^{k}\right),   \nabla_{\vv} f\left(\x^{k}\right) + \bxi_{\vv}^{k+1}  \>.
 \end{eqnarray}

 We can further bound the right hand side  of \eqref{gcs} as follows:
 \begin{eqnarray}\label{temp2}
 &&- \< \eta \nabla g_{\cS\bot}\left(\vv^{k}\right),  \nabla_{\vv} f\left(\x^{k}\right)\> \notag\\
 &=& - \eta \left\| \nabla g_{\cS\bot}\left(\vv^{k}\right)  \right\|^2 -  \< \eta \nabla g_{\cS\bot}\left(\vv^{k}\right),  \nabla_{\vv} f(\x^{k})  - \nabla g_{\cS\bot}\left(\vv^{k}\right) \>\notag\\
 &\leq&-   \frac{7\eta}{8}  \left\|\nabla g_{\cS\bot}\left(\vv^{k}\right) \right\|^2 +  2\eta\left\| \nabla_{\vv} f(\x^{k})  - \nabla g_{\cS\bot}\left(\vv^{k}\right) \right\|^2,
 \end{eqnarray}
 where in the last inequality, we apply:
 \begin{eqnarray}
  \< \eta\nabla g_{\cS\bot}\left(\vv^{k}\right),  \nabla g_{\cS\bot}\left(\vv^{k}\right)  -   \nabla_{\vv} f\left(\x^{k}\right)\>\leq \frac{\eta \left\| \nabla g_{\cS\bot}\left(\vv^{k}\right)\right\|^2}{8} + 2\eta \left\|\nabla g_{\cS\bot}\left(\vv^{k}\right) - \nabla_{\vv} f\left(\x^k\right) \right\|^2.
 \end{eqnarray}
Substituting \eqref{temp2} into \eqref{gcs},  and telescoping the results with $k$ from $0$ to $\sK-1$,  we have
 \begin{eqnarray}\label{g_vv end}
 &&g_{\cS\bot}\left(\vv^{\sK}\right) \\
  &\leq&g_{\cS\bot}\left(\vv^{0}\right)  -  \sum_{k = 1}^{\sK}  \eta\< \nabla g_{\cS\bot}\left(\vv^{k-1}\right),  \bxi_{\vv}^k\>  - \frac{7\eta}{8}\sum_{k = 0}^{\sK-1}  \left\|  \nabla g_{\cS\bot}\left(\x^{k}\right)  \right\|^2+2\eta\sum_{k = 0}^{\sK-1}\left\| \nabla_{\vv} f\left(\x^{k}\right)  - \nabla g_{\cS\bot}\left(\vv^{k}\right) \right\|^2\notag\\
 &\overset{a}{\leq}&g_{\cS\bot}\left(\vv^{0}\right)  -  \sum_{k = 1}^{\sK}  \eta\< \nabla g_{\cS\bot}\left(\vv^{k-1}\right),  \bxi_{\vv}^k\>  - \frac{7\eta}{8}\sum_{k = 0}^{\sK-1}  \left\|  \nabla g_{\cS\bot}\left(\x^{k}\right)  \right\|^2 + \frac{\rho^2 B^4 \eta K_0}{2},\notag
 \end{eqnarray}
 where in $\overset{a}\leq$, we use $$ \|\nabla_{\vv} f\left(\x^{k}\right)  - \nabla g_{\cS\bot}\left(\vv^{k}\right)\|  = \left\|\cP_{\cS\bot} \left(\nabla f\left(\x^{k}\right) -\nabla g\left(\x^{k}\right) \right)\right\|\leq \left\|\nabla f\left(\x^{k}\right) -\nabla g\left(\x^{k}\right) \right\|  \overset{\eqref{basic1}}{\leq} \rho B^2/2,$$
 holds for all $k\leq \sK-1$.

\end{proof}

\subsection{Proofs of Proposition \ref{faster descent}}
With Lemma \ref{analysis on uu} and \ref{analysis on vv} in hand,  the mainly rest to do is  to prove  $$\sum_{k = 0}^{\sK_0-1}  \left\| \nabla g_{\cS\bot}\left(\vv^k\right)  \right\|^2+\sum_{k=0}^{\sK_0-1}\left\|  \nabla  g_{\cS} \left(\y^k\right)\right\|^2 =  \tilde{\Omega}(1)$$ and bound the noise term $ -  \sum_{k = 1}^{\sK}  \eta\< \nabla g_{\cS\bot}\left(\vv^{k-1}\right),  \bxi_{\vv}^k\> $.  	We separately consider   two cases:
\begin{enumerate}
	\item $\left\|\nabla f\left(\x^{0}\right)\right\|> 5 \sigma\asymp 1$,\label{case11}
	\item $\left\|\nabla f\left(\x^{0}\right)\right\|\leq 5 \sigma\asymp 1$. \label{case12}
\end{enumerate}

\begin{enumerate}[(i)]
\item \noindent\textbf{Case \ref{case11}:}   in the sense that the gradient is large,  we show that function value is guaranteed to decrease monotonously.  

\begin{proof}[Proofs of Proposition \ref{faster descent} in Case \ref{case11}]
Because $\left\|\nabla f\left(\x^{0}\right)\right\|> 5 \sigma$,  we have, for all $0\leq k \leq \sK-1$,
\begin{eqnarray}\label{boundna}
\left\|\nabla f\left(\x^{k}\right) \right\|\geq \left\| \nabla f\left(\x^{0}\right) \right\| - \left\|\nabla f\left(\x^{k}\right)-  \nabla f\left(\x^{0}\right) \right\|\overset{a}{\geq} 5\sigma - L B \overset{LB\leq \sigma}{\geq} \frac{9}{2} \sigma,
\end{eqnarray}
where  $\overset{a}\geq$ uses $\| \x^k - \x^0\|\leq B$ for all $k \leq \sK_1$, the $L$-Lipschitz continuous of the gradient. 
Furthermore,  we also  have
\begin{eqnarray}\label{case1all}
&&f\left(\x^{k+1}\right)-f\left(\x^{k}\right) \notag\\
&\leq& \< \nabla f\left(\x^{k}\right),\x^{k+1} - \x^{k}  \> + \frac{L}{2}\| \x^{k+1} - \x^{k}\|^2\notag\\
&\overset{a}{=}& -\eta\< \nabla f\left(\x^{k}\right),\nabla f(\x^k) +\bxi^{k+1} \> + \frac{L\eta^2}{2}\| \nabla f(\x^k) +\bxi^{k+1}\|^2\notag\\
&\overset{b}{\leq}&  -\eta \left\| \nabla f\left(\x^{k}\right) \right\|^2 - \eta  \< \nabla f\left(\x^{k}\right) , \bxi^{k+1}\> + L\eta^2\left\| \nabla f\left(\x^{k}\right) \right\|^2 + L \eta^2\left\|\bxi^{k+1} \right\|^2\notag\\
&\overset{c}{\leq}& -\frac{15\eta}{16}\left\| \nabla f\left(\x^{k}\right) \right\|^2 +\frac{5\eta}{32}\left\| \nabla f\left(\x^{k}\right) \right\|^2 + \frac{8}{5} \eta \sigma^2 + L\eta^2 \sigma^2\notag\\
&\leq& -\frac{25\eta}{32}\left\| \nabla f\left(\x^{k}\right) \right\|^2 +2\eta\sigma^2\overset{\eqref{boundna}}{\leq} - \eta\left(\frac{25}{32} - \frac{8}{81}\right)\left\| \nabla f\left(\x^{k}\right) \right\|^2,
\end{eqnarray}
where $\overset{a}=$ uses the update rule of SGD: $\x^{k+1} =\x^k -\eta \nabla f\left(\x^k\right) -\eta \bxi^{k+1}$, $\overset{b}\leq$ uses $\| \a +\b\|^2 \leq 2\| \a\|^2+ 2\| \b\|^2$,
 in $\overset{c}\leq$, we use $L\eta\leq \frac{1}{16}$ from \eqref{parametersetting}, $-\< \nabla f\left(\x^{k}\right) , \bxi^{k+1}\> \leq \frac{5}{32} \left\| \nabla f\left(\x^{k}\right) \right\|^2+ \frac{8}{5}\left\|\bxi^{k+1} \right\|^2$, and $\left\|\bxi^{k+1}\right\|\leq \sigma$. By telescoping \eqref{case1all} with $k$ from $0$ to $\sK-1$, we have
\begin{eqnarray}\label{D43}
&&f\left(\x^{\cK}\right)-f\left(\x^{0}\right)  \leq  - \eta\left(\frac{25}{32} - \frac{8}{81}\right) \sum_{k=0}^{\sK-1}\left\| \nabla f\left(\x^{k}\right) \right\|^2.
\end{eqnarray}
On the other hand, again by the update rule of SGD, we have
\begin{eqnarray}\label{D44}
&&\left\| \eta \sum_{k=0}^{\sK-1} \nabla f\left(\x^{k}\right)  \right\|=\left\| -\eta \sum_{k=0}^{\sK-1} \nabla f\left(\x^{k}\right)  \right\|\notag\\
&=& \left\| \x^{\sK} - \x^{0} + \eta \sum_{k=1}^{\sK}\bxi^k  \right\|\notag\\
&\geq& \left\| \x^{\sK} - \x^{0}\right\| - \left\| \eta \sum_{k=1}^{\sK}\bxi^k  \right\|.
\end{eqnarray}
 By the Vector-Martingale Concentration Inequality in Theorem \ref{proh}, we have with probability $1 - p/12$,  
\begin{eqnarray}\label{tempp}
\left\| \sum_{k=1}^{\sK} \bxi^k  \right\|\leq \left\| \sum_{k=1}^{K_0} \bxi^k \cdot \cI_{\sK\geq k}  \right\| \overset{a}{\leq} 2\eta \sigma\sqrt{ K_0 \log(48/p)} \leq \frac{B}{16},
\end{eqnarray} 
where $\overset{a}\leq$ uses $ \cI_{ k\leq \sK}$ is measurable on $\cF^{k-1}$ and $\|\bxi^k\|\leq \sigma$.
So if \eqref{tempp} happens,  and  $\x^k$ exits $\cB\left(\x^{0},B\right)$ in $K_0$ iterations, we have
\begin{eqnarray}\label{D46}
&&\eta \sum_{k=0}^{\sK-1}\left\| \nabla f\left(\x^{k}\right) \right\|^2 \notag\\ 
&\overset{a}{\geq}&  \frac{1}{ \eta \sK}\left\| \eta\sum_{k=0}^{\sK-1}  \nabla f\left(\x^{k}\right) \right\|^2 \overset{\eqref{D44}}{\geq}    \frac{1}{ \eta\sK} \left(B - \frac{1}{16}B\right)^2\notag\\
&\geq&\frac{15^2B^2}{16^2 \eta\sK }\overset{\sK \leq K_0}{\geq} \frac{15^2B^2}{16^2\eta K_0},
\end{eqnarray}
where in $\overset{a}\geq$, we use the inequality  that 
$$ \left\|\sum_{i=1}^l \a_i \right\|^2 \leq l \sum_{i=1}^l \left\|\a_i\right\|^2,  $$
holds for all $l\geq 1$.
Plugging \eqref{D46} into \eqref{D43},  with probability at least $1- p/12$ (\eqref{tempp} happens), we have
\begin{eqnarray}
f\left(\x^{\sK_0}\right) \leq  f\left(\x^{0}\right)  - \left(\frac{25}{32} - \frac{8}{81}\right) \frac{15^2B^2}{16^2\eta K_0} \leq  f\left(\x^{0}\right)  -\frac{B^2}{7\eta K_0}.
\end{eqnarray} 
\end{proof}

\noindent\textbf{Case \ref{case12}:}  To obtain the result, we first prepare the following lemmas:
\item We  fuse Lemma \ref{analysis on uu} and \ref{analysis on vv} and obtain the lemma shown below:
\begin{lemma}\label{sufficient descent1}
	With the parameters set in \eqref{parametersetting}, and if $\|\nabla f(\x^0)\|\leq 5\sigma$,  with probability $1-p/4$ (\eqref{probability1}, \eqref{probability2} and \eqref{probability3} happen), we have
\begin{eqnarray}\label{end2}
f\left(\x^{\sK}\right)&\leq& f\left(\x^{0}\right)  -  \eta\sum_{k = 1}^{\sK}  \<  \nabla g_{\cS\bot}\left(\vv^{k-1}\right),  \bxi_{\vv}^k\> + \left(\frac{3}{256}+ \frac{1}{80}\right) \frac{B^2}{\eta K_0}\\
&& - \frac{7\eta}{8}  \sum_{k = 0}^{\sK-1}  \left\| \nabla g_{\cS\bot}\left(\vv^{k}\right)  \right\|^2 -\frac{25\eta}{32}\sum_{k=0}^{\sK-1}\left\|  \nabla  g_{\cS} \left(\y^k\right)\right\|^2  . \notag
\end{eqnarray}
\end{lemma} 

\begin{proof}[Proofs of Lemma \ref{sufficient descent1}]
 Because  $\left\|\nabla g_{\cS\bot}\left(\vv^{0}\right)\right\|=\left\|\nabla_\vv f\left(\x^{0}\right)\right\|\leq \left\|\nabla f\left(\x^{0}\right)\right\|\leq 5 \sigma$,  for all $0\leq k \leq \sK-1$, we have
\begin{eqnarray}\label{boundnabla}
\left\|\nabla g_{\cS\bot}\left(\vv^{k}\right) \right\|\leq \left\| \nabla g_{\cS\bot}\left(\vv^{0}\right) \right\| + \left\|\nabla g_{\cS\bot}\left(\vv^{k}\right)- \nabla g_{\cS\bot}\left(\vv^{0}\right) \right\|\overset{a}{\leq} 5\sigma + L B \overset{LB \leq \frac{1}{2}}{\leq} \frac{11}{2} \sigma,
\end{eqnarray}
where in $\overset{a}\leq$, we use $\left\|\vv^k - \vv^0 \right\| = \left\|\cP_{\cS\bot}(\x^k - \x^0)\right\|\leq B$ and $L$-Lipschitz continuous  gradient for $g_{\cS\bot}(\cdot)$. 
In the same way,  for all $0\leq k \leq \sK-1$,  we  have
\begin{eqnarray}
\left\|\nabla f\left(\x^{k}\right) \right\|\leq \frac{11\sigma}{2},
\end{eqnarray}
which indicates that
\begin{eqnarray}\label{xbound}
\left\|\x^{\sK} - \x^{0} \right\| \leq \left\| \x^{0} -\x^{\sK-1} \right\|+ \left\|\nabla f\left(\x^{\sK-1}\right) + \bxi^{\sK}\right\| \leq B + \frac{13}{2}\eta \sigma \overset{\eqref{parametersetting}}{\leq} B+\frac{B}{100}.  
\end{eqnarray}

We  then bound  the difference between $f(\x^{\sK})$ and $g(\x^{\sK})$: 
 using  $\rho$-smoothness of Hessian,  we have
\begin{eqnarray}\label{tt}
f\left(\x^{\sK}\right) - f\left(\x^{0}\right) - g_{\cS}\left(\uu^{\sK}\right) - g_{\cS\bot}\left(\vv^{\sK}\right)     \leq \frac{\rho}{6}\|\x^{\sK} - \x^{0} \|^3 \overset{\eqref{xbound}}{\leq}  \frac{\rho B^3}{5},
\end{eqnarray}
Then by adding \eqref{uu end3} and \eqref{g_vv end},  using \eqref{tt}, and $g_{\cS}\left(\uu^{0}\right) + g_{\cS\bot}\left(\vv^{0}\right)=0$, we have, with probability at least $1-p/4$, (\eqref{probability1}, \eqref{probability2} and \eqref{probability3} happen)
\begin{eqnarray}\label{end22}
&&f\left(\x^{\sK}\right) \\
&\leq& f\left(\x^{0}\right)  -   \eta\sum_{k = 1}^{\sK} \<  \nabla g_{\cS\bot}\left(\vv^{k-1}\right),  \bxi_{\vv}^k\> + 4\eta \sigma^2 \left( 1+3\log(K_0)\right) \log(48/p)\notag\\
&& - \frac{7\eta}{8} \sum_{k = 0}^{\sK-1}  \left\| \nabla g_{\cS\bot}\left(\vv^{k}\right)  \right\|^2 -\frac{25\eta  }{32} \sum_{k=0}^{\sK-1}\left\|  \nabla  g_{\cS} \left(\y^k\right)\right\|^2 + \frac{3\rho^2 B^4 \eta K_0}{2} +\frac{\rho B^3}{5}. \notag
\end{eqnarray}

With the parameter set in \eqref{parametersetting},
\begin{eqnarray}
4\eta \sigma^2 \left( 1+3\log(K_0)\right) \log(48/p) \leq  \frac{B^2}{256\eta K_0},
\end{eqnarray}
\begin{eqnarray}
\frac{3\rho^2 B^4 \eta K_0}{2}  \leq  \frac{B^2}{128\eta K_0},
\end{eqnarray}
and
\begin{eqnarray}
\frac{\rho B^3}{5} \leq  \frac{B^2}{80\eta K_0},
\end{eqnarray}
 we obtain  with probability at least $1-p/4$, (\eqref{probability1}, \eqref{probability2} and \eqref{probability3} happen)
\begin{eqnarray}
f\left(\x^{\sK}\right)&\leq& f\left(\x^{0}\right)  -  \eta\sum_{k = 1}^{\sK}  \<  \nabla g_{\cS\bot}\left(\vv^{k-1}\right),  \bxi_{\vv}^k\> + \left(\frac{3}{256}+ \frac{1}{80}\right) \frac{B^2}{\eta K_0}\notag\\
&& - \frac{7\eta}{8}  \sum_{k = 0}^{\sK-1}  \left\| \nabla g_{\cS\bot}\left(\vv^{k}\right)  \right\|^2 -\frac{25\eta\sum_{k=0}^{\sK-1}\left\|  \nabla  g_{\cS} \left(\y^k\right)\right\|^2  }{32}, \notag
\end{eqnarray}
implying \eqref{end2}.
\end{proof}

\item Furthermore, the following lemma  ensures the function value sufficient descent:
\begin{lemma}\label{sk2 lemma}
With probability $1-\frac{p}{6}$ (\eqref{probability1} and \eqref{probability4} happen), if    $\x^k$ exits $\cB\left(\x^{\cK},B\right)$ in $K_0$ iterations, we have  
\begin{eqnarray}\label{2endim}
&&\eta \sum_{k = 0}^{\sK-1}  \left\| \nabla g_{\cS\bot}\left(\vv^{k}\right)  \right\|^2+\eta\sum_{k=0}^{\sK-1}\left\|  \nabla  g_{\cS} \left(\y^k\right)\right\|^2\geq  \frac{169B^2}{512\eta K_0}.
\end{eqnarray}
\end{lemma}
\begin{proof}[Proofs of Lemma \ref{sk2 lemma}]
From the update rule of SGD, we have
\begin{eqnarray}
&&\left\|\eta\sum_{k=0}^{\sK-1}  \left(\nabla g_{\cS\bot}\left(\vv^{k}\right) +  \nabla g_{\cS} \left(\y^k\right)\right) \right\|=\left\|-\eta\sum_{k=0}^{\sK-1}  \left(\nabla g_{\cS\bot}\left(\vv^{k}\right) +  \nabla g_{\cS} \left(\y^k\right)\right) \right\|\notag\\
&\overset{a}{=}&\left\|\vv^{\sK} - \vv^{0}   + \eta\sum_{k=0}^{\sK-1} \left(\bxi_{\vv}^{k+1}-  \nabla g_{\cS\bot} (\vv^k)  + \nabla_{\vv} f(\x^k)\right)  + \y^{\sK} - \y^{0} \right\|\notag\\
&\overset{b}\geq&\left\|  \vv^{\sK}  - \vv^{0}  + \eta\sum_{k=0}^{\sK-1} \bxi_{\vv}^{k+1}   +  \left(\uu^{\sK} - \uu^{0}\right) -\left( \z^{\sK} - \z^{0}\right)   \right\|\notag\\ &&-\left\|\eta \sum_{k=0}^{\sK-1} \left(\nabla g_{\cS\bot} \left(\vv^k\right)  - \nabla_{\vv} f\left(\x^k\right)\right) \right\| \notag\\
&\overset{c}\geq&\left\|  \x^{\sK}  - \x^{0}\right\|   - \left\|\z^{\sK} - \z^{0} \right\| -\eta \left\| \sum_{k=1}^{\sK} \bxi_{\vv}^k  \right\| - \frac{\eta K_0 \rho B^2}{2}\notag\\
&\overset{\eqref{parametersetting}}\geq&\left\|  \x^{\sK}  - \x^{0}\right\| - \left\|\z^{\sK} - \z^{0} \right\| -\frac{B}{32} - \eta \left\| \sum_{k=1}^{\sK} \bxi_{\vv}^k \right\|.,
\end{eqnarray}
where  $\overset{a}=$ uses $\vv^k = \vv^{k-1} - \eta\bxi^{k+1}_\vv-\eta \nabla_{\vv} f(\x^k)$ and $\y^{k} = \y^{k-1} - \eta\nabla g_{\cS} \left(\y^k\right) $, $\overset{b}\geq$ uses $\z^{\sK} = \uu^{\sK} - \y^{\sK}$, $\z^0 = \uu^0 = \y^0=\mathbf{0}$, and triangle inequality, $\overset{c}\geq$ uses \eqref{basic1}.

From \eqref{z4}, with probability at least $1- \frac{1}{12}p$, we have  $\left\|\z^{\sK} - \z^{0} \right\|\leq \frac{3B}{32}$.
By the Vector-Martingale Concentration Inequality in Theorem \ref{proh}, we have with probability $1 - p/12$,  
\begin{eqnarray}\label{probability4}
\left\| \eta\sum_{k=1}^{\sK} \bxi_{\vv}^k \right\| =\left\|  \eta\sum_{k=1}^{K_0} \left(\bxi_{\vv}^k\cdot \cI_{ k\leq \sK}\right)\right\|\overset{a}{\leq} 2\eta \sigma\sqrt{ K_0 \log(48/p)} \overset{\eqref{parametersetting}}{\leq} \frac{B}{16},
\end{eqnarray} 
where $\overset{a}\leq$ uses $ \cI_{ k\leq \sK}$ is measurable on $\cF^{k-1}$ and $\|\bxi_{\vv}^k\|\leq \sigma$.
We obtain
\begin{eqnarray}\label{2end2}
\left\|\eta\sum_{k=0}^{\sK-1}  \left( \nabla g_{\cS\bot} \left(\vv^k\right) +  \nabla g_{\cS} \left(\y^k\right)\right)\right\|\geq\left\|  \x^{\sK}  - \x^{0}\right\| - \frac{3B}{16}. 
\end{eqnarray}
So with probability $1-p/6$, if    $\x^k$ exits $\cB\left(\x^{\cK},B\right)$ in $K_0$ iterations,   we  have
\begin{eqnarray}
&&\eta \sum_{k = 0}^{\sK-1}  \left\| \nabla g_{\cS\bot}\left(\vv^{k}\right)  \right\|^2+\eta\sum_{k=0}^{\sK-1}\left\|  \nabla  g_{\cS} \left(\y^k\right)\right\|^2 \notag\\ 
&\overset{a}{\geq}&  \frac{1}{2 \eta\sK}\left\| \eta \sum_{k = 0}^{\sK-1} \left(  \nabla g_{\cS\bot}\left(\vv^{k}\right)  +   \nabla  g_{\cS} \left(\y^k\right)\right)  \right\|^2 \notag\\
&\overset{\eqref{2end2}}{\geq}&\frac{169B^2}{512 \eta\sK}\overset{\sK \leq K_0}{\geq} \frac{169B^2}{512\eta K_0},
\end{eqnarray}
where $\overset{a}\geq$ uses
$$ \left\|\sum_{i=1}^l \a_i \right\|^2 \leq l \sum_{i=1}^l \left\|\a_i\right\|^2,  $$
holds for all $l\geq 1$.
\end{proof}

\item Now, we have all the ingredients necessary to prove Proposition \ref{faster descent}:
\begin{proof}[Proofs of Proposition \ref{faster descent} in Case \ref{case12}]
We first bound the noise term $\sum_{k = 1}^{\sK} \<  \nabla g_{\cS\bot}\left(\vv^{k-1}\right),  \bxi_{\vv}^k\>$.  We have for all $k$ from $1$ to $K_0$
\begin{eqnarray}
\EE   \left[\eta\<\nabla g_{\cS\bot}\left(\vv^{k-1}\right),  \bxi_{\vv}^k\> \cdot\cI_{k\leq \sK}\mid \cF^{k-1}\right] = \mathbf{0},
\end{eqnarray}	
 From \eqref{boundnabla}, and $\|\bxi_{\vv}^k \|\leq \sigma$, we have
\begin{eqnarray}
\left|-\eta\<\nabla g_{\cS\bot}\left(\vv^{k-1}\right),  \bxi_{\vv}^k\> \cdot\cI_{k\leq \sK}\right|\leq \frac{11\eta\sigma^2}{2},
\end{eqnarray}
and 
\begin{eqnarray}
\EE   \left|\eta\<\nabla g_{\cS\bot}\left(\vv^{k-1}\right),  \bxi_{\vv}^k\> \cdot\cI_{k\leq \sK}\mid \cF^{k-1}\right|^2\leq   \eta^2 \sigma^2\cI_{k\leq \sK} \left\|g_{\cS\bot}\left(\vv^{k-1}\right) \right\|^2
\end{eqnarray}
by Data-Dependent Berinstein inequality in Theorem  \ref{proh2} with $\delta =\frac{p}{3\log(K_0)}$, we have with probability at least $1 - \frac{p}{3}$,
\begin{eqnarray}\label{probability6}
 &&\sum_{k=1}^{K_0} \left\{-\eta\<\nabla g_{\cS\bot}\left(\vv^{k-1}\right),  \bxi_{\vv}^k\> \cdot\cI_{k\leq \sK}\right\}\\
 &\leq& \max\left\{11\eta\sigma^2\cdot\log\left(\frac{3\log(K_0)}{p}\right), 4\sqrt{\eta^2\sigma^2\sum_{k=0}^{\sK-1}\left\|\nabla g_{\cS\bot}\left(\vv^{k}\right) \right\|^2   }\cdot \sqrt{\log\left(\frac{3\log(K_0)}{p}\right)} \right\} .\notag 
\end{eqnarray}
With the parameter set in \eqref{parametersetting}, we have
\begin{eqnarray}\label{add2}
11\eta\sigma^2\cdot\log\left(\frac{3\log(K_0)}{p}\right)\leq \frac{B^2}{100\eta K_0},
\end{eqnarray}
and 
\begin{eqnarray}\label{add1}
&&4\eta\sqrt{\sigma^2\sum_{k=0}^{\sK-1}\left\|\nabla g_{\cS\bot}\left(\vv^{k-1}\right) \right\|^2   }\cdot \sqrt{\log\left(\frac{3\log(K_0)}{p}\right)}\notag\\
&\overset{a}{\leq}& 32\log\left(\frac{3\log(K_0)}{p}\right)\eta\sigma^2 + \frac{\eta}{8}\sum_{k=0}^{\sK-1}\left\|\nabla g_{\cS\bot}\left(\vv^{k}\right) \right\|^2 \notag\\
&\overset{\eqref{parametersetting}}{\leq}&   \frac{B^2}{32\eta K_0} +\frac{\eta}{8}\sum_{k=0}^{\sK-1}\left\|\nabla g_{\cS\bot}\left(\vv^{k}\right) \right\|^2,
\end{eqnarray}
where  $\overset{a}\leq$ uses $\sqrt{ab}\leq \frac{a+b}{2}$ for $a\geq0$.
Substituting \eqref{add2} and \eqref{add1} into \eqref{probability6},  with probability at least $1-3/p$, we have
\begin{eqnarray}\label{add4}
 &&\sum_{k=1}^{\sK} \left\{-\eta\<\nabla g_{\cS\bot}\left(\vv^{k-1}\right),  \bxi_{\vv}^k\>\right\}\leq \frac{B^2}{32\eta K_0} +\frac{\eta}{8}\sum_{k=0}^{\sK-1}\left\|\nabla g_{\cS\bot}\left(\vv^{k}\right) \right\|^2.
\end{eqnarray}
Fusing \eqref{add4} with \eqref{end2} in Lemma \ref{sufficient descent1}, using $\frac{7}{8} - \frac{1}{8}\leq \frac{25}{32}$, we have with probability  at least $1- \frac{7}{12}p$ (\eqref{probability1}, \eqref{probability2}, \eqref{probability3}, and \eqref{probability6} happen),
\begin{eqnarray}
f\left(\x^{\sK}\right)&\leq& f\left(\x^{0}\right)  -   \left(\frac{3}{256}+ \frac{1}{80}+ \frac{1}{32}\right) \frac{B^2}{\eta K_0} - \frac{3\eta}{4}  \sum_{k = 0}^{\sK-1}  \left\| \nabla g_{\cS\bot}\left(\vv^{k}\right)  \right\|^2 -\frac{3\eta}{4}\sum_{k=0}^{\sK-1}\left\|  \nabla  g_{\cS} \left(\y^k\right)\right\|^2. \notag
\end{eqnarray}
	
Finally, applying Lemma \ref{sk2 lemma}, if $\x^k$ moves out of the ball in $K_0$ iteration,  with probability at least $1-\frac{2}{3}p$ (\eqref{probability1}, \eqref{probability2}, \eqref{probability3}, \eqref{probability4}, and \eqref{probability6} happen),  we have 
\begin{eqnarray}
f\left(\x^{\sK_0}\right)  - f\left(\x^{0}\right)\leq -\left(  \frac{3}{4}\cdot\frac{169}{512}  -\frac{3}{256}- \frac{1}{80}- \frac{1}{32} \right) \frac{B^2}{\eta K_0}\leq -\frac{B^2}{7\eta K_0}.
\end{eqnarray}

Combining Case \ref{case11} and Case \ref{case12}, we obtain Proposition \ref{faster descent}.
\end{proof}

\end{enumerate}

\section{Deferred Proofs of Part \mbox{III}: Finding SSP}\label{proofthree}
\begin{proof}[Proofs of Proposition \ref{local convergence}]
Clearly, under the random event $\cH_0$ in Part \mbox{I} happens,  we know that if $\lambda_{\min} \nabla f(\x^{0}) \leq  -\delta_2$,   $\x^{k}$ must gone out of the ball. Thus with probability at least $1-p/3$ (the random events $\cH_0$ in Part \mbox{I} happens),  if $\x^k$ does not move out the ball in $K_0$ steps,  we have $\lambda_{\min} \left(\nabla f(\x^{0})\right) \geq  -\delta_2$.  Using  that $f(\x)$ has continuous Hessian, we have 
\begin{equation}
\!\lambda_{\min} \left(\nabla f(\bbx)\right) \geq \lambda_{\min} \left(\nabla f(\x^{0})\right) - \rho \left\| \bbx - \x^{0}  \right\|_2 \geq -\delta_2 - \frac{\rho}{K_0}\sum_{k=0}^{K_0-1} \left\| \bbx - \x^{0}  \right\|  \geq -\frac{17}{16} \delta_2  =-17 \delta.
\end{equation}

To a give upper bound on the $\|\nabla f(\bbx)\|^2$, we follow the idea by considering quadratic   approximations in Part \mbox{II}. We have
\begin{eqnarray}
\left\| \nabla g ( \bbx)  \right\| &\overset{a}{=}&   \left\|\frac{1}{K_0}\sum_{k=0}^{K_0-1}  \nabla g(\x^k)  \right\| \notag\\
&\leq& \left\|\frac{1}{K_0}\sum_{k=0}^{K_0-1}  \nabla f(\x^k)  \right\|+ \frac{1}{K_0}\sum_{k=0}^{K_0-1} \left\| \nabla f(\x^k)  - \nabla g(\x^k)  \right\|  \notag\\
&=&\frac{1}{K_0\eta}\left\| \x^{K_0-1}  - \x^{0} - \eta \sum_{k=1}^{K_0} \bxi^k \right\|+ \frac{1}{K_0}\sum_{k=0}^{K_0-1}\left\| \nabla f(\x^k)  - \nabla g(\x^k)  \right\| \notag\\
&\overset{\eqref{basic1}}{\leq}&\frac{1}{K_0\eta}\left\| \x^{K_0-1}  - \x^{0} \right\| + \frac{1}{K_0}\left\| \sum_{k=1}^{K_0} \bxi^k \right\| + \frac{\rho}{2}B^2\notag\\
&\leq& \frac{B}{K_0 \eta} + \frac{\rho B^2}{2} + \frac{1}{K_0}\left\|  \sum_{k=1}^{K_0} \bxi^k \right\|\notag\\
&\leq&\left(\frac{16}{\tC_1}+ \frac{1}{2} \right) \rho B^2+ \frac{1}{K_0}\left\|  \sum_{k=1}^{K_0} \bxi^k \right\|,
\end{eqnarray}
where in $\overset{a}\leq$, we use the gradient of the quadratic function $g(\cdot)$ is a linear mapping.

By the Vector-Martingale Concentration Inequality, we have with probability $1 - 2p/3$,  
\begin{eqnarray}\label{probability5}
\frac{1}{K_0} \left\| \sum_{k=1}^{K_0} \bxi^k\right\|\leq 2 \sigma\sqrt{ K_0 \log(6/p)}/K_0\leq \rho B^2.
\end{eqnarray}

Using $\| \bbx - \x^{0}\| \leq B$, we have $\left\| \nabla f ( \bbx)  \right\| \leq \left\| \nabla g ( \bbx)  \right\| + \frac{\rho B^2}{2}\leq 18\rho B^2  $.

In all, we have with probability at least $1-p$,  $\lambda_{\min} \left(\nabla f(\bbx)\right)  \geq  - 17\delta$, and $\left\| \nabla f ( \bbx)  \right\|  \leq 18\rho B^2$.
\end{proof}

\begin{proof}[Proofs of Theorem \ref{theo:main}]	
By union bound, with probability at least $1 - T_1\cdot p$,  if at step $T_0 = T_1 \cdot K_0$, Algorithm  \ref{algo:SGD} has not stopped,  $\x^k$ must have moved out of the ball  at least $T_1$ times, then from Proposition \ref{faster descent}, the function values shall decrease at least 
 $$T_1 \cdot  \frac{B^2}{7\eta K_0} \geq \Delta +  \frac{B^2}{7\eta K_0} >\Delta.  $$
Contradiction with Assumption \ref{assu:main}. Thus with probability at least $1 - T_1\cdot p$,  Algorithm \ref{algo:SGD} shall stop before $T_0 $ steps. Further, fusing with Proposition \ref{local convergence},  we have  with probability at least $1 - (T_1+1)\cdot p$, Algorithm  \ref{algo:SGD}  outputs a second-order stationary point satisfying \eqref{ESSP} in $T_0$ steps.
	
\end{proof}

\section{Proof of Proposition \ref{prop:noise}}\label{sec:proof,prop:noise}
\begin{proof}[Proof of Proposition \ref{prop:noise}]
	\begin{enumerate}[(i)]
		\item
		Recall the multivariate gaussian noise $\tilde\bxi = \sigma/\sqrt{d} * \bchi$ where $\bchi \sim N(0, \Ib_d)$.
		We show that it satisfies \eqref{smallprob1}. Clearly, it satisfies \eqref{bound2}.
		
		Let $\v$ be an arbitrary unit vector, and due to symmetry in below we assume WLOG $\v = \e_1$.
		Recall we have set $\cA$ satisfying the $(q^*,\v)$-narrow property in Definition \ref{defi:qnarrow}.
		Then
		$$
		\left\{\u + q \e_1: \u\in \cA,\
		q \in \left[ q^*, \infty \right)
		\right\}
		\subseteq \cA^c
		.
		$$
		If set $\cA$ contains no points of $\u,\u+q\e_1$ for each $q\ge q^*$, then $\cA[\bullet,\a_{\backslash 1}] := \{a_1: (a_1,\a_{\backslash 1})^\top \in \cA\}$ is a subset of $\RR$ and has Lebesgue measure $\le 1.1q^*$.
		This is because that for any given $\a_{\backslash 1}=(a_2,\dots,a_d)$ there exists an $a_1^*$ such that $(a_1^*,\a_{\backslash 1})^\top\in \cA$ and we pick $a_1^*$ to be the infimum of such.  Then it is easy to conclude that $(a_1^*+q,\a_{\backslash 1})^\top \in \cA^c$ for any $q > 1.1q^*$, and that
		$$
		\cA[\bullet,\a_{\backslash 1}]
		\subseteq
		[a_1^*, a_1^* + 1.1q^*]
		.
		$$
		Therefore we have for any $\cA$ admitting $(q^*,\v)$-narrow property where $q^* = (\sigma/4\sqrt{d})$, that for any given $\bchi_{\backslash 1}$,
		\begin{align*}
		\PP(\sigma/\sqrt{d} * \bchi_1 \in \cA[\bullet,\bchi_{\backslash 1}] \mid \bchi_{\backslash 1}) 
		&\le
		\frac{1}{\sqrt{2\pi}} \int_{(4q^*)^{-1} \cA[\bullet,\bchi_{\backslash 1}]} \exp(-z^2/2) dz
		\\&\le
		\frac{1.1q^*}{4q^*} \cdot \frac{1}{\sqrt{2\pi}}
		<
		\frac14
		,
		\end{align*}
		where $\cA[\bullet,\bchi_{\backslash 1}]$ is of Lebesgue measure $\le 1.1q^*$.
		Taking expectation again gives
		$$
		\PP(\sigma/\sqrt{d} * \bchi \in \cA)
		=
		\EE\left[ \PP(\sigma/\sqrt{d} * \bchi_1 \in \cA[\bullet,\bchi_{\backslash 1}] \mid \bchi_{\backslash 1}) \right]
		\le \frac14
		,
		$$
		and we complete the proof that $\bxi = \sigma/\sqrt{d} * \bchi$ is $\v$-disperse for any $\v$.

		\item
		For example, recall the uniform ball-shaped noise $\tilde\bxi = \sigma * \bxi_b$, where $\bxi_b$ is uniformly sampled from $\cB^d$, the unit ball centered at ${\bf 0}$.
		We prove that \eqref{smallprob1} holds in this case.
		Assume once again that $\v = \e_1$ because of symmetry.
		Using classical results in Multivariate Calculus (or see \citet{jin2017escape}) and $(q^*,\v)$-narrow property property in Definition \ref{defi:qnarrow} of set $\cA^*$ we have
		\beq\label{Aprop_prim}
		\PP(\sigma * \bxi_b \in \cA)
		=
		\frac{Vol_d((\sigma^{-1}\cA) \cap \cB^d)}{Vol_d(\cB^d)}
		\le
		\frac{q^*}{\sigma}\cdot \frac{Vol_{d-1}(\cB^{d-1})}{Vol_d(\cB^d)}
		.
		\eeq
		It is well known that the $d$-dimensional unit ball $\cB^d$ of $\RR^d$ has volume being
		$$
		Vol_d(\cB^d)
		=
		\frac{\pi^{d/2}}{\Gamma\left(\frac{d}{2} + 1\right)}
		,
		$$
		and analogously for $\cB^{d-1}$.
		We have
		\begin{align*}
		\frac{Vol_{d-1}(\cB^{d-1})}{Vol_d(\cB^d)}
		&=
		\frac{\pi^{(d-1)/2}}{\Gamma\left(\frac{d-1}{2} + 1\right)} \cdot
		\frac{\Gamma\left(\frac{d}{2} + 1\right)}{\pi^{d/2}}
		=
		\frac{\Gamma\left(\frac{d+1}{2}+\frac12\right)}{\pi^{1/2} \Gamma\left(\frac{d+1}{2}\right)}
		\le
		\sqrt{\frac{d+1}{2\pi}}
		\le
		\sqrt{d}
		,
		\end{align*}
		where we applied a well-known fact that $\Gamma(x+1/2) \le \Gamma(x) \sqrt{x}$ for all $x > 0$.
		Plugging in the definition $q^* := \sigma / 4\sqrt{d}$ in \eqref{Aprop_prim}, we have proved \eqref{smallprob1} that $\tilde\bxi$ is $\v$-disperse for any $\v$.
			\item
		For stochastic gradients injected by artificial, dispersive noise, we prove that the $\v$-disperse property still holds.
		Let $\tilde{\bm\gamma}$ be some artificial noise that has the $\v$-dispersive property, that is, for an arbitrary set $\cA$ with $(q^*,\v)$-narrow property, where $q^* = \sigma / 4\sqrt{d}$.
		Then as in Definition \ref{defi:qnarrow} one has, by the linearly scalable property after Definition \ref{defi:qnarrow}, that $\PP\left( \tilde{\bm\gamma} \in \cA - \g \right) \le 1/4$ for any fixed vector $\g\in\RR^d$.
		Then we have by injecting such independent noise to the stochastic gradient $\nabla f(\w;\bzeta)$ that
		\begin{align*}
		\PP\left( \nabla f(\w;\bzeta) + \tilde{\bm\gamma} \in \cA \mid \nabla f(\w;\bzeta) \right)
		&=
		\PP\left( \tilde{\bm\gamma} \in \cA - \nabla f(\w;\bzeta)  \,\big|\, \nabla f(\w;\bzeta) \right) 
		\le
		\frac14
		,
		\end{align*}
		where in the last step we used the independence of $\tilde{\bm\gamma} \in \cA$ and $\nabla f(\w;\bzeta)$.
		Taking expectation in the last line gives
		\beq\label{smallprob_pri}
		\PP\left( \nabla f(\w;\bzeta) + \tilde{\bm\gamma} \in \cA  \right)
		\le \frac14
		,
		\eeq
			so \eqref{smallprob1} is satisfied for this noise-injected stochastic gradient $\nabla f(\w;\bzeta) + \tilde{\bm\gamma}$.

	\end{enumerate}

\end{proof}

\end{document}